\documentclass[11pt]{amsart}
\usepackage{mathrsfs} 
\usepackage[margin=1.25in]{geometry}
\usepackage{ifthen}
\usepackage{graphicx}
\usepackage{epsfig}
\usepackage{dsfont, amssymb}
\usepackage{amsfonts,dsfont,amssymb,amsthm,stmaryrd,bbm}
\usepackage{amsmath}

\usepackage{xcolor}
\usepackage{hyperref}

\usepackage[toc,page]{appendix} 
\usepackage{hyperref,mathtools}
\hypersetup{colorlinks=true,pdftitle=""
}

\let\originallesssim\lesssim
\let\originalgtrsim\gtrsim

\DeclareRobustCommand{\lesssim}{%
  \mathrel{\mathpalette\lowersim\originallesssim}%
}
\DeclareRobustCommand{\gtrsim}{%
  \mathrel{\mathpalette\lowersim\originalgtrsim}%
}

\makeatletter
\newcommand{\lowersim}[2]{%
  \sbox\z@{$#1<$}%
  \raisebox{-\dimexpr\height-\ht\z@}{$\m@th#1#2$}%
}
\makeatother

\newtheorem{thm}{Theorem}[section]

\newtheorem{remark}[thm]{Remark}
\newtheorem{lem}[thm]{Lemma}
\newtheorem{prop}[thm]{Proposition}

\newtheorem{defn}[thm]{Definition}
\newtheorem{cor}[thm]{Corollary}

\newcommand\independent{\protect\mathpalette{\protect\independent}{\perp}} 
\def\independent#1#2{\mathrel{\rlap{$#1#2$}\mkern2mu{#1#2}}}

\newcommand{\R}{\mathbb{R}}

\def\phi{\varphi}

\def\bee{\begin{eqnarray*}}
\def\ene{\end{eqnarray*}}
\usepackage{mathtools}

\title{\small Sharp Weighted Discrete 
$p$-Hardy Inequality and Stability}
\author{Ali Barki} 
\date{} 

\begin{document}
\thanks{The author is supported by the Labex MME-DII, funded by ANR (reference ANR-11-LBX-0023-01), and by the Fondation Simone et Cino Del Duca, France. This research has been conducted within the FP2M federation (CNRS FR 2036).}

\address{MODAL'X, UPL, Univ. Paris Nanterre, CNRS, F92000 Nanterre France}

\address{Department of Mathematics, Imperial College London,
180 Queen’s Gate, SW7 2AZ, London, \phantom{...} United Kingdom}

\email{ali.barki@ens-rennes.fr}
\maketitle
\begin{abstract} In this paper,
we prove a $p$-Hardy inequality on the discrete half-line with weights $n^{\alpha}$ for all real $p > 1$. Building on the work of Miclo for $p = 2$ and Muckenhoupt in the continuous settings, we develop a quantitative approach for the existence of a $p$-Hardy inequality involving two measures $\mu$ and $\nu$ on the discrete half-line. We also investigate the comparison between sharp constants in the discrete and continuous settings and explore the stability of the inequality in the discrete case.
\end{abstract} 
 
\section{Introduction}
\phantom{.}\\
The classical Hardy inequality states that, for every non-negative sequence
\((a_n)_{n\ge 1}\) and every \(p\in(1,\infty)\),
\[
\sum_{n=1}^{\infty}
\left(
\frac{1}{n}\sum_{i=1}^{n} a_i
\right)^p
\le
\left(\frac{p}{p-1}\right)^p
\sum_{n=1}^{\infty} a_n^p .
\]
Although this inequality had already appeared implicitly in Hardy's proof of
Hilbert's double series theorem \cite{Har20}, Landau gave in his 1921 letter to
Hardy \cite{Lan21} an early proof and showed that the constant
\(\left(\frac{p}{p-1}\right)^p\) is sharp. Later, Hardy obtained weighted
continuous analogues of this inequality, under suitable assumptions on
\(\alpha\in\mathbb{R}\) and \(p\in(1,\infty)\), in the form \\ 
$\bullet$ for $ \alpha < p-1 $:
$$
\int_0^\infty y^{\alpha-p} \left(  \int_0^y f(x) \, dx \right)^p \, dy 
\leq \left(  \frac{p} {p -1 - \alpha} \right)^p \int_0^\infty f(x)^p \ x^{\alpha} \, dx,
$$

$\bullet$ when $ \alpha > p-1$, the inequality becomes:
$$
\int_0^\infty y^{\alpha-p} \left( \int_y^\infty f(x) \, dx \right)^p \, dy 
\leq \left(  \frac{p} { \alpha+1-p} \right)^p  \int_0^\infty f(x)^p \ x^{\alpha} \, dx,
$$
where, $f$ is a measurable and non-negative function on $(0, \infty).$ In both cases, the constants involved are sharp, ensuring optimal bounds for these weighted inequalities. These two inequalities are unified into the following single form:
\begin{equation}
\int_{0}^\infty |\phi'(x)|^p \, x^{\alpha} \, dx \geq \left|  \frac { \alpha+1-p}{p}\right|^p \int_{0}^\infty |\phi(x)|^p \, x^{\alpha-p} \, dx
\label{eq 1.1} \tag{1.1}
\end{equation}
for all $\phi \in C_0^\infty(0, \infty)$, where $\alpha \neq p-1$. It was subsequently demonstrated that the case $\alpha = p-1$ is critical, where the two quantities cannot be controlled by a positive constant. However, we have an alternative Leray-type inequality:
\begin{equation}
\int_{1}^{\infty} |\phi'(x)|^p x^{p-1}\, dx \geq \left( \frac{p-1}{p} \right)^p \int_{1}^{\infty} \frac{|\phi(x)|^p}{x \log^p(x)} \, dx, \tag{1.2} \label{eq 1.2}
\end{equation}
for all $\phi \in C_0^\infty(1, \infty)$. In this paper, we investigate discrete versions of \eqref{eq 1.1} for all finitely supported sequences $(u_{n})_{n\geq 0}$ with $u_{0} = 0:$

\begin{equation}
\sum_{n=1}^{\infty} |u_{n} - u_{n-1}|^{p} \, n^{\alpha} \geq \mathcal{A}_{\alpha,p} \sum_{n=1}^{\infty} |u_{n}|^{p} \, n^{\alpha - p},
\tag{1.3} \label{eq 1.3}
\end{equation}
where $p > 1,$  $\alpha \in \mathbb{R},$ and $\mathcal{A}_{\alpha,p}$ a positive constant for which the inequality holds. \\ \\
A key motivation for this work lies in the observation that inequality \eqref{eq 1.3} has been predominantly studied in the literature only for the specific case $p = 2$. This observation emphasizes the necessity of extending the analysis to the general case. The case $p = 2$ was first explored by Gupta \cite{Gup22}, who established the inequality for $\alpha \in [0, 1) \cup [5, \infty)$ with an approach based on the supersolution method (see \cite{Caz20} for an idea of the approach in continuum). Huang and Ye \cite{Ye23} further extended the analysis for any real \(\alpha\) when $p=2.$ In a different direction,
Liu \cite{Liu12} considered the general case \(p>1\), but only in the
range \(\alpha<p-1\), which is precisely the restriction that we aim to
remove here. The two \(p=2\) approaches both rely on determining a family of weights $g_{\alpha}(n)$ satisfying:
$$\sum_{n=1}^{\infty} |u_{n} - u_{n-1}|^{2} \, n^{\alpha} \geq \sum_{n=1}^{\infty} |u_{n}|^{2} \, g_{\alpha}(n),$$

this reduces to proving that
 $\displaystyle g_{\alpha}(n) \geq \frac{(\alpha - 1)^2}{4} n^{\alpha - 2}.$ For $p = 2,$ the latest inequality presents significant challenges, and the difficulty increases significantly for a general $p > 1,$ (see \cite{FKP23} for an improved discrete $p$-Hardy inequality, corresponding to the case $\alpha=0).$ In this context, we introduce a method that simplifies the proof of inequality \eqref{eq 1.3} for all \(p > 1\). This is achieved through an estimate for the discrete Hardy inequalities involving two general measures \(\mu, \nu\) on the discrete half-line. This estimate is directly adapted from Muckenhoupt \cite{Muck72} for the analogous continuous result and from Miclo \cite{Mic99} for the discrete case when \(p = 2\). A variant of this result for rooted trees has been explored in the literature. We refer to the appendix of \cite{Mic98}, where these inequalities are derived from the work of Evans, Harris, and Pick in \cite{Evans95} for continuous trees. \\ \\
Under the assumption of existence, we denote \(\mathcal{A}^{disc}_{\alpha,p}\) as the sharp constant for inequality \eqref{eq 1.3}. For all \(\alpha \neq  p - 1\), the corresponding sharp constant for the continuous version  \eqref{eq 1.1} is given by: \(\mathcal{A}^{cont}_{\alpha,p} = \left| \frac{\alpha + 1 - p}{p} \right|^p\). It follows that the sharp constant $\mathcal{A}^{disc}_{\alpha,p}$ exists if and only if $\alpha \neq p - 1$. We further show in Section \ref{sec 4} that for all $\alpha \geq 0$ and $\alpha \neq p - 1$, the discrete and continuous sharp constants are equal:  
$$ \mathcal{A}^{disc}_{\alpha,p} = 
 \mathcal{A}^{cont}_{\alpha,p}.   
$$
In the critical case $\alpha=p-1$,
we establish a discrete Leray-type inequality
for sequences in $H_0^2(\mathbb{N})$,
with the same constant as in the continuous inequality \eqref{eq 1.2}.\\ \\
Section \ref{sec 3} is devoted to the construction of an optimal $p$-Hardy weight associated with a quasi-linear Schrödinger operator on an infinite graph. The results are applied to subsets of the discrete half-line. This concept was first introduced for $p = 2$ in \cite{KPP18} and later extended to $p > 1$ in \cite{Fis22}. The notion of optimality is fundamentally linked to the theory of criticality, as developed in \cite{KPP20,KPP18} for the discrete case and \cite{DP16} in the continuous case. In Section \ref{sec 4}, we investigate
the case $\alpha<0$ and establish two-sided bounds for the discrete sharp
constant in terms of its continuous counterpart. In particular, we prove that
for every $p>1$ and $\alpha\leq-1$,
\[
\mathcal{A}_{\alpha,p}^{\text{disc}}
<
\mathcal{A}_{\alpha,p}^{\text{cont}}.
\]
Thus, unlike the case $\alpha\geq0$,
the discrete and continuous sharp constants
do not always coincide.
Finally, we demonstrate that the method not only provides a sharp constant for $\alpha \geq 0$, but also allows the derivation of additional results. The stability of the inequality and the subcriticality of the associated functional energy are among the key contributions of this work. \\ \\ 
We now recall the notion of stability for functional inequalities. Stability refers to the fact that an element nearly attaining equality in a functional inequality is quantitatively close to one of its minimizers. For further detailed analysis of this concept, we refer to \cite{Car17,Fig12,Fig09,Fig10,FigInd13,FigMagPra13}. \\ \\
Let $X$ be a real normed linear space, and let \(\mathcal{E}, \mathcal{F}: X \rightarrow \mathbb{R}\cup \{\infty\} \) be two functionals satisfying the inequality
\begin{equation}
\mathcal{E}(x) \leq \mathcal{F}(x) \quad \text{for every} \quad x \in X.
\tag{1.4} \label{eq 1.4}
\end{equation}
The inequality \eqref{eq 1.4} is called sharp if, for any \(\lambda < 1\), there exists an element \(x \in X\) such that \(\mathcal{E}(x) > \lambda \mathcal{F}(x)\).The set of optimizers  $X_{\text{opt}}$ for inequality \eqref{eq 1.4} is defined as:
$$
X_{\text{opt}} = \{x \in X \mid \mathcal{E}(x) = \mathcal{F}(x) \text{ and } \mathcal{F}(x) < \infty\}.
$$
The inequality is considered optimal if \(X_{\text{opt}}\) is non-empty, It is worth noting that every optimal inequality is sharp; however, the converse does not necessarily hold. Assume that the inequality \eqref{eq 1.4} is optimal, and let \(\{x_{n}\}_{n\geq 0}\) be a sequence in \(X\) such that: $$\lim\limits _{n\to\infty}(\mathcal{F}(x_{n})-\mathcal{E}(x_{n}))=0.$$
We aim to identify conditions under which \eqref{eq 1.4} implies
$\lim\limits_{n\to\infty}d(x_{n}, X_{\text{opt}})=0,$ where $d$ is a metric on \(X\) that may differ from the norm metric. Recall that a rate function is a mapping $\Psi: [0, \infty) \rightarrow [0, \infty)$ that is strictly convex and vanishes at the origin. The functional inequality \eqref{eq 1.4} is called $(d, \Psi)$-stable if, for all $x \in X:$
\begin{equation}
\mathcal{F} (x) - \mathcal{E}(x) \geq \Psi(d(x, X_{\text{opt}})) \tag{1.5} \label{eq 1.5}
\end{equation}
 
The functional inequality
\eqref{eq 1.4} is considered $\Psi$ stable if the condition \eqref{eq 1.5} holds for some metric $d$ on $X.$ For $n \geq 3$, the Sobolev inequality on $\mathbb{R}^n$ is given by:
\begin{equation}\left(\int_{\mathbb{R}^{n}}|u(x)|^{2n/(n-2)}\,dx\right)^{(n-2)/n} \leq S_{2,n} \int_{\mathbb{R}^{n}}|\nabla u(x)|^{2}\,dx, \tag{1.6} \label{eq 1.6}
\end{equation}
 
It has been shown in \cite{Bian18} that the inequality \eqref{eq 1.6} is $\Psi$-stable with $\Psi: t \to  k_{n} t^2$ for some $k_{n} > 0$. In this paper, we establish that inequality \eqref{eq 1.3} with $\mathcal{A}_{\alpha,p} = \mathcal{A}_{\alpha,p}^{\text{disc}}$ is $\Psi$-stable with $\Psi: t \to t^{p}$, for all non-negative \( \alpha \) such that \( \alpha \neq p-1 \). \\ 
We fix $p > 1$ as a real number and introduce the following notations, which will be consistently used in the subsequent sections. Let $\mathcal{M}(\mathbb{N}^{*})$ denote the set of positive measures defined on the discrete half-line, and $F(\mathbb{N}^{*})$ the space of functions mapping $\mathbb{N}^{*}$ to $\mathbb{R}$. The subspace $F^{0}(\mathbb{N}^{*})$ consists of functions $f \in F(\mathbb{N}^{*})$ satisfying $\displaystyle \sum_{n=1}^{\infty} f(n) = 0$. Additionally, $C_{c}(\mathbb{N})$ denotes the space of sequences with finite support, and $H^{M}_{0}(\mathbb{N}_{0})$ is defined as the subspace of $C_{c}(\mathbb{N})$ consisting of sequences that vanish on the first $M$ indices, where $M \in \mathbb{N}^{*}$: \\
\[
H^{M}_{0}(\mathbb{N}_{0}) := \{ u \in C_{c}(\mathbb{N}) \mid u_{0} = \cdots = u_{M-1} = 0 \}.
\]
\section{A discrete version of Muckenhoupt's criterion}
The aim of this section is to provide a discrete form of the Muckenhoupt criterion related to the existence of a weighted $p$-Hardy inequality for prescribed weights (see \cite{Muck72}). The results presented generalize Miclo's work \cite{Mic99} for $p=2$ extending the method to a wider range of $p$-values and additional relevant cases. To proceed with the analysis, we first introduce the notations that will be used throughout this section.\\ \\
Let $\mu$ and $\nu$ two positive measures on $\mathbb{N}^{*}.$ For every $f \in F(\mathbb{N}^{*})$
our interest lies in the following inequality:
\begin{align}
\sum_{x \in \mathbb{N}^{*}} \left| \ \sum_{y=1}^{x} f(y) \  \right|^{p} \mu(x) \leq  C_{\mu,\nu} \left( \sum_{x \in \mathbb{N}^{*}} |f(x)|^p \nu(x)\right)  \tag{2.1} \label{eq 2.1}
\end{align}
where $C_{\mu,\nu} \in (0,+\infty]$ denotes the smallest constant associated with the inequality stated above over $F(\mathbb{N}^{*}).$ On the other hand, we define
$C_{\mu,\nu}^{0} \in (0,+\infty]$ as the smallest constant corresponding to the inequality restricted to the subspace $F^{0}(\mathbb{N}^{*}).$ Furthermore, we introduce two constants, $ B_{\mu,\nu}^{(1)}$ and $ B_{\mu,\nu}^{(2)},$ that are relevant to the forthcoming propositions:

$$ B_{\mu,\nu}^{(1)} = \sup_{r \geq 1} \left[ \left( \sum_{x = r}^{\infty} \mu(x) \right) \cdot \left( \sum_{x = 1}^{r} \frac{1}{\left[ \nu(x) \right]^{\frac{1}{p-1}}} \right)^{p-1} \right]$$ 
$$ B_{\mu,\nu}^{(2)} = \sup_{r \geq 1} \left[ \left( \sum_{x=1}^{r} \mu(x) \right) \cdot \left( \sum_{x =r+1}^{\infty} \frac{1}{\left[ \nu(x) \right]^{\frac{1}{p-1}}} \right)^{p-1} \right]$$ 
An initial result characterizes the finiteness of $C_{\mu,\nu}$ and provides a quantitative bound for its control.

\begin{prop} \label{Prop 2.1}
$C_{\mu,\nu} < \infty$ if and only if $B_{\mu,\nu}^{(1)}< \infty.$ In this context, the following bounds hold:
$$ B_{\mu,\nu}^{(1)} \leq C_{\mu,\nu} \leq   \frac{p^{p}}{(p-1)^{p-1}} B_{\mu,\nu}^{(1)} $$
\end{prop}

\begin{proof} Under the assumption 
$B_{\mu,\nu}^{(1)} < \infty,$ consider the following function, defined on $\mathbb{N}^{*}:$

$$h:x \to \displaystyle \sum_{y = 1}^{x} \frac{1}{\left[ \nu(y) \right]^{\frac{1}{p-1}}}$$
For $f \in F(\mathbb{N}^{*}),$ applying H\"older's inequality yields:

$$ \left| \sum_{y=1}^{x} f(y)  \right| \leq  \sum_{y=1}^{x} |f(y)| 
\leq \left( \sum_{y=1}^{x} |f(y)|^p \nu(y) [h(y)]^{\frac{p-1}{p}} \right)^{\frac{1}{p}}. \left( \sum_{y = 1}^{x} \frac{1}{\left[\nu(y)\right]^{\frac{1}{p-1}}  \left[h(y)\right]^{\frac{1}{p}}} \right)^{\frac{p-1}{p}}$$
As a result, we deduce that:

$$ \mathop{\sum_{x \in \mathbb{N}^{*}}} \left| \ \sum_{y=1}^{x} f(y) \  \right|^{p} \mu(x) \leq  \sum_{x \in \mathbb{N}^{*}} \mu(x) \left( \sum_{y=1}^{x} |f(y)|^p \nu(y) [h(y)]^{\frac{p-1}{p}} \right). \left( \sum_{y = 1}^{x} \frac{1}{\left[\nu(y)\right]^{\frac{1}{p-1}}  \left[h(y)\right]^{\frac{1}{p}}} \right)^{p-1}$$

Using Fubini-Tonelli, the right-hand side simplifies to:

$$\sum_{y \in \mathbb{N}^{*}} |f(y)|^{p}\nu(y) [h(y)]^{\frac{p-1}{p}}   \left[ \sum_{x \geq y} \mu(x) \left( \sum_{s = 1}^{x} \frac{1}{\left[\nu(s)\right]^{\frac{1}{p-1}}  \left[h(s)\right]^{\frac{1}{p}}} \right)^{p-1} \right]
$$
In addition, for $0\leq a<b$ and $m\geq1,$ we observe that $\frac{b-a}{b^{1-\frac{1}{m}}} \leq m (b^{\frac{1}{m}}- a^{\frac{1}{m}}).$ Consequently, for $m=\frac{p}{p-1},$ under the convention $h(0)=0$ it follows that:

$$
\sum_{s=1}^{x} \frac{1}{\left[\nu(s)\right]^{\frac{1}{p-1}}  \left[h(s)\right]^{\frac{1}{p}}} \leq   \frac{p}{p-1} \sum_{s=1}^{x} ([h(s)]^{\frac{p-1}{p}}-  [h(s-1)]^{\frac{p-1}{p}}) = \frac{p}{p-1} [h(x)]^{\frac{p-1}{p}} 
$$
As for $m=p:$

$$\sum_{x\geq y} \frac{\mu(x)}{ \left[\displaystyle \sum_{t\geq x} \mu(t) \right] ^{\frac{p-1}{p}}} \leq p  \mathop{\sum_{x\geq y}}  (\mu \,[x,\infty[)^{\frac{1}{p}} - (\mu\, [x+1,\infty[)^{\frac{1}{p}} = p \left(\sum_{x\geq y} \mu (y) \right)^{\frac{1}{p}}
$$
Building on the definition of
$B_{\mu,\nu}^{(1)},$ we revisit the inequality:
$$
\sum_{y \in \mathbb{N}^{*}} |f(y)|^{p}\nu(y) [h(y)]^{\frac{p-1}{p}}   \left[ \sum_{x \geq y} \mu(x) \left( \sum_{s = 1}^{x} \frac{1}{\left[\nu(s)\right]^{\frac{1}{p-1}}  \left[h(s)\right]^{\frac{1}{p}}} \right)^{p-1} \right] $$

$$\leq  \left( \frac{p}{p-1}\right)^{p-1} \sum_{y \in \mathbb{N}^{*}} |f(y)|^{p}\nu(y) [h(y)]^{\frac{p-1}{p}}   \left( \sum_{x \geq y} \mu(x) \, [h(x)]^{\frac{(p-1)^2}{p}} \right)$$
Applying the upper bound for $h$ in terms of $B_{\mu,\nu}^{(1)},$ we obtain:
$$\leq \left( \frac{p}{p-1}\right)^{p-1} [B_{\mu,\nu}^{(1)}]^{\frac{p-1}{p}}  \sum_{y \in \mathbb{N}^{*}} |f(y)|^{p}\nu(y) \,[h(y)]^{\frac{p-1}{p}} 
\mathop{\sum_{x\geq y}} \frac{\mu(x)}{\left[\displaystyle \sum_{t\geq x} \mu(t) \right] ^{\frac{p-1}{p}}}$$

$$\leq p \left( \frac{p}{p-1}\right)^{p-1}[B_{\mu,\nu}^{(1)}]^{\frac{p-1}{p}}  \sum_{y \in \mathbb{N}^{*}} |f(y)|^{p}\nu(y) [h(y)]^{\frac{p-1}{p}}
\left( \mathop{\sum_{x\geq y}} \mu(x) \right)^{\frac{1}{p}}$$
By reorganizing the terms and applying the bound on $\mu([y,\infty[)$ it follows that:
$$\leq  p \left( \frac{p}{p-1}\right)^{p-1} [B_{\mu,\nu}^{(1)}]^{\frac{p-1}{p}}  [B_{\mu,\nu}^{(1)}]^{\frac{1}{p}}  \sum_{y \in \mathbb{N}^{*}} |f(y)|^{p}\nu(y) = \frac{p^{p}}{(p-1)^{p-1}}  B_{\mu,\nu}^{(1)}   \sum_{y \in \mathbb{N}^{*}} |f(y)|^{p}\nu(y) $$
From the above estimates, we conclude that: $C_{\mu,\nu} < \infty$ and $ C_{\mu,\nu} \leq \frac{p^{p}}{(p-1)^{p-1}}  B_{\mu,\nu}^{(1)}.$  \\ \\
Assume now that $C_{\mu,\nu} < \infty .$ Fix $r \in \mathbb{N}^{*}$ and set 
$f: x \mapsto \left[\frac{1}{\nu(x)}\right]^{\frac{1}{p-1}} \mathbf{1}_{x\le r},$ hence we obtain:

$$\left(\sum_{x \geq r} \mu(x) \right)\left| \sum_{x=1}^{r} \frac{1}{[\nu(x)]^{\frac{1}{p-1}}} \right|^p\leq \sum_{x \in \mathbb{N}^{*}} \left|\sum_{y=1}^{x} f(y) \right|^{p} \mu(x) \leq  C_{\mu,\nu}\sum_{x \in \mathbb{N}^{*}} |f(x)|^p \nu(x) = C_{\mu,\nu} \sum_{x=1}^{r} \frac{1}{[\nu(x)]^{\frac{1}{p-1}}}$$
As a result
$$\left(\sum_{x=r}^{\infty} \mu(x) \right)
\left( \sum_{x=1}^{r} \frac{1}{[\nu(x)]^{\frac{1}{p-1}}} \right)^{p-1} \leq  C_{\mu,\nu} $$
 \\
By considering the supremum over $r \in\mathbb{N}^{*},$ we obtain: $B_{\mu,\nu}^{(1)} \leq C_{\mu,\nu} < \infty$ which completes the proof.
\end{proof}
 
\begin{remark} \label{rmk 2.2} As shown in \cite{Ye23} and \cite{Gup22}, for $u\in H_{0}^{1}(\mathbb{N})$ and $\alpha \geq 5,$ the following inequality holds:
$$ \sum_{n=1}^{\infty} |u_{n}-u_{n-1}|^2 n^{\alpha} \geq \frac{(\alpha-1)^2}{4}  \sum_{n=1}^{\infty} |u_n|^2 n^{\alpha-2}$$
However, if we apply Proposition \ref{Prop 2.1} with $f : n \mapsto u_n - u_{n-1}$, $\mu(n) = n^{\alpha-2}$, $\nu(n) = n^{\alpha}$, and $p = 2$, we find that $C_{\mu,\nu} = \infty$. This arises from the fact that the constant $C_{\mu,\nu}$ operates on $F(\mathbb{N}^{*}),$ while the function $f$ chosen here belongs to $F^0(\mathbb{N}^*)$, a more restricted class.
\end{remark}
The main result of the next proposition lies in the fact that, under the finiteness condition of one of the constants $B_{\mu,\nu}^{(i)}$ for $i=1,2,$ we also obtain the finiteness in addition to a bound control of $C_{\mu,\nu}^{0},$ which operates on $F_{0}(\mathbb{N}^{*}).$
\begin{prop} \label{prop 2.3} Under the assumption that $\min(B_{\mu,\nu}^{(1)},B_{\mu,\nu}^{(2)}) < \infty,$ we have $C_{\mu,\nu}^{0}<\infty$. Moreover, we have the following bound:
$$C_{\mu,\nu}^{0} \leq   \frac{p^{p}}{(p-1)^{p-1}} \min(B_{\mu,\nu}^{(1)},B_{\mu,\nu}^{(2)}) $$
\end{prop}
\begin{proof}
First, observe that if $B_{\mu,\nu}^{(2)}= \infty,$ the result follows directly from Proposition \ref{Prop 2.1}. \\
Next, if $B_{\mu,\nu}^{(2)}< \infty,$ we define the function $h$ on $\mathbb{N}^{*},$ as follows:

$$h: x \to \displaystyle \sum_{y = x}^{\infty} \frac{1}{\left[ \nu(y) \right]^{\frac{1}{p-1}}}$$
We proceed with the same reasoning as in the proof of the last proposition.\\  \\
Let $f\in F^{0}(\mathbb{N}^{*}),$ using H\"older's inequality, we obtain:
$$ \left| \sum_{y=1}^{x} f(y)  \right|= \left| \sum_{y=x+1}^{\infty} f(y)  \right| \leq  \sum_{y=x+1}^{\infty} |f(y)| 
\leq \left( \sum_{y=x+1}^{\infty} |f(y)|^p \nu(y) [h(y)]^{\frac{p-1}{p}}  \right)^{\frac{1}{p}}. \left( \sum_{y = x+1}^{\infty} \frac{1}{\left[\nu(y)\right]^{\frac{1}{p-1}}  \left[h(y)\right]^{\frac{1}{p}}} \right)^{\frac{p-1}{p}}
$$
As a result, we get:
$$ \mathop{\sum_{x \in \mathbb{N}^{*}}} \left| \ \sum_{y=1}^{x} f(y) \  \right|^{p} \mu(x) \leq  \sum_{x \in \mathbb{N}^{*}} \mu(x) \left( \sum_{y=x+1}^{\infty} |f(y)|^p \nu(y) [h(y)]^{\frac{p-1}{p}}  \right). \left( \sum_{y = x+1}^{\infty} \frac{1}{\left[\nu(y)\right]^{\frac{1}{p-1}}  \left[h(y)\right]^{\frac{1}{p}}} \right)^{p-1}$$

$$ \leq  \sum_{x \in \mathbb{N}^{*}} \mu(x) \left( \sum_{y=x}^{\infty} |f(y)|^p \nu(y) [h(y)]^{\frac{p-1}{p}}  \right). \left( \sum_{y = x+1}^{\infty} \frac{1}{\left[\nu(y)\right]^{\frac{1}{p-1}}  \left[h(y)\right]^{\frac{1}{p}}} \right)^{p-1}
$$
Applying Fubini--Tonelli, we obtain
\[\sum_{x=1}^{\infty}
\mu(x)
\left|\sum_{y=1}^{x}f(y)\right|^p \leq
\sum_{y=2}^{\infty}
|f(y)|^p\nu(y)
h(y)^{\frac{p-1}{p}}
\sum_{x=1}^{y-1}
\mu(x)
\left(
\sum_{s=x+1}^{\infty}
\frac{1}{
\nu(s)^{\frac{1}{p-1}}
h(s)^{\frac1p}}
\right)^{p-1}.
\]
Since $h(s)-h(s+1)=
\nu(s)^{-\frac{1}{p-1}},$
the concavity of
$t\mapsto t^{\frac{p-1}{p}}$
implies
\[\sum_{s=x+1}^{\infty}
\frac{1}{
\nu(s)^{\frac{1}{p-1}}
h(s)^{\frac1p}}\leq
\frac{p}{p-1}
\sum_{s=x+1}^{\infty}
\left(
h(s)^{\frac{p-1}{p}}
-
h(s+1)^{\frac{p-1}{p}}
\right)=
\frac{p}{p-1}
h(x+1)^{\frac{p-1}{p}}.
\]
For every $x\geq1$, set
$$M_x=\sum_{t=1}^{x}\mu(t)
\qquad M_0=0.$$
By the definition of $B_{\mu,\nu}^{(2)}$, one have 
$M_x \,h(x+1)^{p-1}
\leq B_{\mu,\nu}^{(2)}.$ Consequently,
\[\sum_{x=1}^{y-1}
\mu(x)
h(x+1)^{\frac{(p-1)^2}{p}}
\leq
[B_{\mu,\nu}^{(2)}]^{\frac{p-1}{p}}
\sum_{x=1}^{y-1}
\frac{\mu(x)}{M_x^{\frac{p-1}{p}}}\leq
p[B_{\mu,\nu}^{(2)}]^{\frac{p-1}{p}}
M_{y-1}^{\frac1p}.
\]
Here, the last inequality follows from
\[
\frac{\mu(x)}{M_x^{\frac{p-1}{p}}}
\leq
p\left(
M_x^{\frac1p}
-
M_{x-1}^{\frac1p}
\right).
\]
Moreover,
$M_{y-1}\,h(y)^{p-1}
\leq B_{\mu,\nu}^{(2)}.$
Combining these estimates gives
\[\sum_{x=1}^{\infty}
\mu(x)
\left|\sum_{y=1}^{x}f(y)\right|^p \leq
p\left(\frac{p}{p-1}\right)^{p-1}
B_{\mu,\nu}^{(2)}
\sum_{y=2}^{\infty}|f(y)|^p\nu(y) \leq
\frac{p^p}{(p-1)^{p-1}}
B_{\mu,\nu}^{(2)}
\sum_{y=1}^{\infty}|f(y)|^p\nu(y).
\]
This completes the proof.
\end{proof} 

As a consequence, we obtain the following corollary, proving the existence of the discrete analogues of the inequalities \ref{eq 1.1} and \ref{eq 1.2} in the continuous settings.
\begin{cor} \label{cor 2.4} Let $\alpha \in \mathbb{R}$ and $u \in H_{0}^{1}(\mathbb{N}).$ Then, there exists a constant
$\mathcal{A}_{\alpha}>0$ such that the following inequality holds:
$$\sum_{n=1}^{\infty} |u_{n}-u_{n-1}|^p \, n^{\alpha} \geq  \mathcal{A}_{\alpha} \sum_{n=1}^{\infty} |u_n|^p \mu_{\alpha}(n)$$
where
$$\mu_{\alpha}(n) = 
\begin{cases}
n^{\alpha-p} & \text{if } \alpha \neq p-1 \\
\mathbf{1}_{\{1\}}(n) + n^{-1} \log^{-p}(n) \mathbf{1}_{[2,\infty[}(n)& \text{if } \alpha=p-1
\end{cases}$$
Moreover, for $\alpha =p-1,$ there exists no such constant $\mathcal{A}>0$  satisfying the inequality:
$$\sum_{n=1}^{\infty} |u_{n}-u_{n-1}|^p \, n^{p-1} \geq  \mathcal{A} \sum_{n=1}^{\infty} \frac{|u_n|^p}{n} $$
\end{cor}
\begin{proof} Let $f: n \to u_{n}-u_{n-1},$ so that $f \in F^{0}(\mathbb{N}^{*}).$ Furthermore, we have the asymptotic equivalences:\\ \\
For $\alpha <p-1,$ it holds that:
$$\left( \sum_{x \geq r} x^{\alpha-p} \right) \cdot \left( \sum_{x = 1}^{r} \frac{1}{x^{\frac{\alpha}{p-1}}} \right)^{p-1} \underset{r \to \infty}{\sim}  \left( \frac{r^{\alpha-p+1}}{p-1-\alpha} \right)    \left( \frac{(p-1)^{p-1} \, r^{p-1-\alpha}}{(p-1-\alpha)^{p-1}} \right) = \frac{(p-1)^{p-1}}{(p-1-\alpha)^{p}} $$ 
Similarly, for $\alpha > p-1,$ we have:
$$\left( \sum_{x=1}^{r} x^{\alpha-p} \right) \cdot \left( \sum_{x = r+1}^{\infty} \frac{1}{x^{\frac{\alpha}{p-1}}} \right)^{p-1} \underset{r \to \infty}{\sim}  \left( \frac{r^{\alpha-p+1}}{\alpha-p+1} \right)    \left( \frac{(p-1)^{p-1} \, r^{p-1-\alpha}}{(\alpha-p+1)^{p-1}} \right) = \frac{(p-1)^{p-1}}{(\alpha-p+1)^{p}} $$ 
Finally, for $\alpha=p-1,$ we obtain:

$$\left( \sum_{x=1}^{r} \, \frac{1}{x} \right)^{p-1} \cdot \left( \sum_{x = r+1}^{\infty} \frac{1}{x \log^p(x) } \right) \underset{r \to \infty}{\sim} \frac{1}{p-1} $$
In all cases, we have the following:
$ \min(B_{\mu,\nu}^{(1)}, B_{\mu,\nu}^{(2)}) < \infty $ for
$\mu(n)=\mu_{\alpha}(n), \nu(n)=n^{\alpha}$ and using Proposition \ref{prop 2.3}, this concludes the proof of the first point. It remains to prove the non-existence of a constant $\mathcal{A}>0$ such that for all $u \in H_{0}^{1}(\mathbb{N}),$ the following holds:

$$\sum_{n=1}^{\infty} |u_{n}-u_{n-1}|^p \, n^{p-1} \geq  \mathcal{A} \, \sum_{n=1}^{\infty} \frac{|u_n|^p} {n}$$
Let $N\geq 1$ be an integer, and define $u_{n}= \mathbf{1}_{[1,N]}(n) + (2-\frac{n}{N}) \mathbf{1}_{[N+1,2N]} (n).$ Thus we have:
$$ \sum_{n=1}^{\infty} \frac{|u_n|^p}{n} \geq  \sum_{n=1}^{N} \frac{1}{n}\underset{N \to \infty}{=} \mathcal{O}(\log(N))
\, \, \, \text{and} \, \, \sum_{n=1}^{\infty} |u_{n}-u_{n-1}|^p \, n^{p-1} =  1+\, \frac{1}{N^p}\sum_{n=N+1}^{2N} n^{p-1} \underset{N \to \infty}{=} \mathcal{O}(1)
$$
Taking the limit as $N\to \infty$ leads to a contradiction, which completes the proof.
\end{proof}
\section{\label{sec 3} Hardy weights for $p$-Schrödinger operators on graphs}
In this section, we investigate the 
$p$-Hardy inequality in the space $H_0^1(\mathbb{N})$ involving two positive measures $\mu$ and $\nu$. Our approach is based on the criticality theory for quasi-linear $p$-Schrödinger operators on general weighted graphs. As a starting point, we present some fundamental and widely recognized results concerning graph structures and Schrödinger operators. For more details and in-depth discussions, we refer the reader to  \cite{Fis24}, and \cite{Ver21} for a detailed analysis in continuous settings.

\subsection*{Weighted Graphs and $p$-Schrödinger Operators} Let $X$ be a countably infinite set. A weighted graph $b$ on $X$ is a symmetric function $b: X \times X \to [0, \infty)$ with a zero diagonal that is locally summable, that is,

$$
\deg(x) := \sum_{y \in X} b(x, y) < \infty \quad \text{for all } x \in X.
$$
The elements of $X$ are called vertices. Two vertices $x$ and $y$ are said to be adjacent or neighbors if $b(x, y) > 0$, denoted by $x \sim y$. A graph $b$ is connected if, for every pair of vertices $x$ and $y$ in $X$, there exists a sequence of vertices $x_0, x_1, \dots, x_n$ in $X$ such that $x_0 = x$, $x_n = y$, and $x_i \sim x_{i+1}$ for all $i = 0, 1, \dots, n-1$. A graph $b$ on $X$ is locally finite if, for each $x \in X$, the set of adjacent vertices has finite cardinality:

$$\#\{y \in X : y \sim x\} < \infty.$$
For the remainder of this section, we assume that $b$ is a locally finite, connected graph on $X$. We denote by $C(V)$ the set of all real-valued functions on $V \subseteq X$, and by $C_c(V)$ the subset of functions with compact support on $V$. Additionally, we define the linear difference operator $\nabla$ on $C(X)$ as:

$$\nabla_{x,y} f := f(x) - f(y), \quad x, y \in X, \, f \in C(X).$$
We now delve into Schrödinger operators. Consider \(p > 1\). For \(1 < p < 2\), we adopt the convention \(\infty \cdot 0 = 0\). Then, for any \(p > 1\) and \(z \in \mathbb{R}\), we define:
$$ z^{(p-1)} := |z|^{p-1} \operatorname{sgn}(z) = |z|^{p-2} z.
$$
Here, the function $\operatorname{sgn} $ is defined as:
$$\operatorname{sgn}: z \to \mathbf{1}_{(0,+\infty)}(z) -\mathbf{1}_{(-\infty,0)}(z).$$ 
For \(p > 1\) and a potential \(q \in C(X)\), the quasi-linear $p$-Schrödinger operator $H$ defined on $C(X)$ as follows:
$$
Hf\,(x) := Lf(x) + q(x) (f(x))^{(p-1)}, \quad  x \in X,
$$
\phantom{.}\\
where $L: C(X) \to C(X)$ is given by \\
$$
Lf(x) := \sum_{y \in X} b(x, y) (\nabla_{x,y} f)^{(p-1)}, \quad  x \in X.
$$
The operator $L$ is called the $p$-Laplacian. A function $u \in C(X)$ is $p$-harmonic (resp. $p$-superharmonic) on $V \subseteq X$ with respect to $H$ if $Hu=0$ (resp, $Hu \geq 0$) on $V.$ For the same potential $q,$ the $p$-energy functional $h : C_c(X) \to \mathbb{R}$ is defined as:
$$
h(\varphi) := \frac{1}{2} \sum_{x,y \in X} b(x, y) |\nabla_{x,y}\varphi|^p + \sum_{x \in X} q(x) |\varphi(x)|^p.
$$
In the expression defining 
$h,$ the first summation is often referred to as the kinetic energy component, while the second summation represents the potential energy component. In the specific case where 
$p=2$ the functional is a quadratic form, commonly known as the Schrödinger form. Let us now focus on the estimation of functionals. For any function $w \in C(X)$, a canonical $p$-functional $\widetilde{w}$ is defined on $C_c(X)$ as follows:
$$\widetilde{w}(\varphi) := \sum_{x \in X} |\varphi(x)|^p w(x), \quad \varphi \in C_c(X).
$$  
For any $p > 1$, given a subset $V \subseteq X$ and a non-negative function $w$ defined on $V$, we define the weighted $\ell^p$-space $\ell^p(V, w)$ as:
$$
\ell^p(V, w) := \left\{ f \in C(X) : \sum_{x \in V} |f(x)|^p w(x) < \infty \right\}.
$$
The Hardy inequality can be written as $h - \widetilde{w} \geq 0$ on $C_c(X)$ for some non-negative weight $w$ on $X$. Here, $\widetilde{w}$ represents the functional, which is introduced to differentiate it from the original function $w$. Next, we proceed to define the fundamental concepts of criticality, subcriticality, and null-criticality, which is a defining feature in this section. A more extensive discussion of these notions, along with additional references, can be found in \cite{KPP20}. \\

\begin{defn} The energy functional $h$, associated with the $p$-Schrödinger operator $H$, is subcritical in $X$ if there exists a positive function $w \in C(X)$ such that $h - \widetilde{w} \geq 0$ on $C_c(X)$. Otherwise, $h$ is critical in $X.$
\end{defn}
\phantom{.}\\
In \cite{Fis22}, it was shown that the functional $h$ is critical if and only if there exists a unique positive superharmonic function. This function is harmonic and corresponds to the Agmon ground state of $h$. This observation naturally leads to the following definition:
\begin{defn}
Let $h$ be a critical energy functional. We call $h$ null-critical with respect to a non-negative function $w$ in $X$ if the Agmon ground state is not in $\ell^p(X, w)$, and otherwise we call it positive-critical with respect to $w$ in $X$.
\end{defn}

Next, we introduce the notion of optimality for a Hardy weight as follows: A non-negative function $w$ is called an optimal $p$-Hardy weight for the energy functional $h$ in $X$ if:
\begin{itemize}
    \item[$\bullet$] $h - \tilde{w}$ is critical in $X$,
    \item[$\bullet$] $h - \tilde{w}$ is null-critical with respect to $w$ in $X$,
    \item[$\bullet$] $h - \tilde{w} \geq \lambda \tilde{w}$ fails to hold on $C_c(X \setminus K)$ for all $\lambda > 0$ and finite $K \subset X$, i.e, $w$ is optimal near infinity for $h$.
\end{itemize}
\phantom{.}\\
Proceeding further, a function $f \in C(X)$ that is positive on $V \subseteq X$ is called \textit{proper} on $V$ if $f^{-1}(I) \cap V$ is a finite set for any compact set $I \subset (0, \infty)$.
Note that a function \( f \in C(X) \) is proper when no subset is explicitly specified, meaning it is proper on the entire space \( X \).
Alternatively, $f \in C(X)$ is said to have bounded oscillation on $X$ if $$ \sup_{x, y \in X, x \sim y} \frac{|f(x)|}{|f(y)|} < \infty.
$$
We define by $\mathcal{B}_{\text{osc}}(X)$ the set of all functions on $X$ with bounded oscillation.
Now we are interested in the following theorem, which is a weak formulation of \cite{Fis24}[Theorem 2.3], and generalizes the result of \cite{KPP18}[Thereom 3.1/4.1] stated for $p = 2$:

\begin{thm} \label{thm 3.3}
Let $b$ be a connected graph and $q \geq 0$ be a finitely supported potential. Suppose that $u$ is a positive $H$-superharmonic function that is $H$-harmonic outside of a finite set, and satisfies:
\begin{itemize}
    \item[(a)] $u : X \to (0, \infty)$ is proper,
    \item[(b)] $\sup_{x, y \in X, x \sim y} \frac{u(x)}{u(y)} < \infty$.
\end{itemize}
Let $h$ be the energy functional associated to $H$. Then $ w := \frac{H [u^{\frac{p-1}{p}}]}{u^{\frac{(p-1)^2}{p}}},
$ is an optimal $p$-Hardy weight for $h$ on $X$.
\end{thm}
\phantom{.}
\subsection*{Restriction to the discrete half-line}
To each positive measure \(\nu\) on \(\mathbb{N}^{*}\), we associate the functional $\mathcal{G}_{\nu}$, defined as follows: \\
$$
\mathcal{G}_{\nu}:= n \to 
\begin{cases}
\displaystyle
\sum_{k=1}^{n} \left[ \frac{1}{\nu(k)} \right]^{\frac{1}{p-1}} \mathbf{1}_{n \geq 1} & \text{if } \displaystyle \sum_{k=1}^{\infty} \left[ \frac{1}{\nu(k)} \right]^{\frac{1}{p-1}} = \infty, \\ 
\displaystyle \sum_{k > n}^{\infty} \left[ \frac{1}{\nu(k)} \right]^{\frac{1}{p-1}} \mathbf{1}_{n \geq 1} & \text{if } \displaystyle \sum_{k=1}^{\infty} \left[ \frac{1}{\nu(k)} \right]^{\frac{1}{p-1}} < \infty. 
\end{cases}
$$

We define a graph structure on \(\mathbb{N}\) using the function \( b_{\nu}: \mathbb{N} \times \mathbb{N} \) given by:
$$b_{\nu}(n,m) = \sum_{k=0}^{\infty} \nu(k+1) \left(\mathbf{1}_{k,k+1}(n,m) + \mathbf{1}_{k+1,k}(n,m) \right)$$

Here the $p$-laplacian $L_{\nu}$ is given by:

$$L_{\nu}f(x) := \sum_{y \in \mathbb{N}} b_{\nu}(x, y) (\nabla_{x,y} f)^{(p-1)}, \quad  x \in \mathbb{N}.$$
In what follows, $h_{\nu}$ denotes the energy functional associated with the $p$-Schrödinger operator $L_{\nu}$. To apply Theorem \ref{thm 3.3}, we consider
$X=\mathbb{N}^{*}$, the restricted graph
$b_{\nu}|_{X\times X}$, and the potential
$q=\nu(1)\mathbf{1}_{\{1\}}$.
The corresponding energy agrees with
$h_{\nu}$ under the Dirichlet condition
$u_0=0$. Moreover, $\mathcal{G}_{\nu}$
is positive, proper, and superharmonic
for the resulting $p$-Schr\"odinger operator,
and harmonic outside $\{1\}$. In this regard, 
we introduce the following sets:

$$\mathcal{M}_{\text{osc}}(\mathbb{N}^{*}) = \left\{ \nu \in \mathcal{M}(\mathbb{N}^{*}) \mid \mathcal{G}_{\nu} \in B_{\text{osc}}(\mathbb{N}^{*}) \right\}.$$

 $$ \mathcal{M}_{1}(\mathbb{N}^{*}) = \left\{ \nu \in \mathcal{M}(\mathbb{N}^{*}) \mid 0 < \liminf_{n \to \infty} \frac{\nu(n+1)}{\nu(n)} \leq \limsup_{n \to \infty} \frac{\nu(n+1)}{\nu(n)} < \infty \right\}.$$
We now introduce the following lemma, which establishes strict inclusions among the defined sets:

\begin{lem} \label{lem 3.4} The following strict inclusions hold:  
$$
\mathcal{M}_{1}(\mathbb{N}^{*}) \subsetneq \mathcal{M}_{\text{osc}}(\mathbb{N}^{*})$$
\end{lem}

\begin{proof}
By the definition of
$\mathcal{M}_{1}(\mathbb{N}^{*})$, the condition implies that there exists $M > 1$ such that for all $n\geq 1$:  
$$
\frac{1}{M} \leq \frac{\nu(n+1)}{\nu(n)} \leq M.
$$
In both cases, regardless of whether the series 
$$\sum_{n \geq 1} \left( \frac{1}{\nu(n)} \right)^{\frac{1}{p-1}}  
$$  
converges or diverges, we have: 
$$
\forall n \in \mathbb{N}^{*}, \quad \max \left( \frac{\mathcal{G}_{\nu}(n+1)}{\mathcal{G}_{\nu}(n)}, \frac{\mathcal{G}_{\nu}(n)}{\mathcal{G}_{\nu}(n+1)} \right) \leq 1 + M^{\frac{1}{p-1}}.  
$$  
which implies that $\nu \in \mathcal{M}_{\text{osc}}(\mathbb{N}^{*})$. To verify that the inclusion is strict, we observe that the measure $\nu(n)$ defined  by
$$\nu(2n) = n \quad \text{and} \quad \nu(2n+1) = 1 $$  
belongs to $\mathcal{M}_{\text{osc}}(\mathbb{N}^{*})$ but not to $\mathcal{M}_{1}(\mathbb{N}^{*})$, as required.
\end{proof}
 
From this point, let $\nu \in \mathcal{M}_{\text{osc}}(\mathbb{N}^{*})$. Let $w_{\nu}$ denote the optimal Hardy weight associated with $h_{\nu}$, obtained using Theorem \eqref{thm 3.3}. We aim to determine the conditions on $\mu \in \mathcal{M}(\mathbb{N}^{*})$ under which we can ensure, or alternatively rule out, the existence of a constant $C>0$ such that, for all $\phi \in H_{0}^{1}(\mathbb{N}):$
\begin{equation}
\sum_{n=1}^{\infty} |\phi(n)-\phi(n-1)|^p \nu(n) \geq C   \sum_{n=1}^{\infty} |\phi(n)|^p \mu(n)
\tag{3.1} \label{eq 3.1}
\end{equation}
To this end, we have the following result:
\begin{prop} \label{prop 3.5} If \, $ \displaystyle \liminf_{n \to \infty} \frac{w_{\nu}(n)}{\mu(n)} > 0,$ 
then inequality \eqref{eq 3.1} holds with a constant \( C > 0 \). On the other hand, if  $\displaystyle \lim_{n \to \infty} \frac{w_{\nu}(n)}{\mu(n)} = 0,$ 
there exists no constant \( C > 0 \) for which \eqref{eq 3.1} holds.
\end{prop}
\begin{proof}
Since $\nu\in
\mathcal{M}_{\text{osc}}(\mathbb{N}^{*})$,
the Dirichlet realization introduced above
allows us to apply Theorem \ref{thm 3.3}
to $h_{\nu}$. Thus, $w_{\nu}$ is an optimal $p$-Hardy weight for $h_{\nu}$, and the condition states that  $\frac{w_{\nu}(n)}{\mu(n)}$ is bounded below by a positive constant, hence the first point. For the second point, through Theorem \ref{thm 3.3}, the weight $w_{\nu}$ being an optimal $p$-Hardy weight for $h_{\nu}$, in particular $w_{\nu}$ is optimal near infinity for $h_{\nu}$. Suppose there exists a constant $C > 0$ for which \eqref{eq 3.1} holds: Let $\epsilon >0$ Under the assumption there exists $n \geq n_{0}$ for which for every $n\geq n_{0}:$  $\displaystyle \frac{w_{\nu}(n)}{\mu(n)} < \frac{C}{1+\epsilon},$ it follows that for all $\phi \in H_0^{n_0}(\mathbb{N}):$
$$\sum_{n=1}^{\infty} |\phi(n)-\phi(n-1)|^p \nu(n) \geq C   \sum_{n=1}^{\infty} |\phi(n)|^p \mu(n) \geq (1+\epsilon) \sum_{n=1}^{\infty} |\phi(n)|^p w_{\nu}(n).
$$
This contradicts the optimality near infinity of $w_{\nu}$ with respect to $h_{\nu}$, thus completing the proof of the proposition.
\end{proof}

\begin{remark} \label{rmk 3.6} For the special case $\nu(n) = n^{\alpha}$, from Lemma \ref{lem 3.4} we get $\nu \in \mathcal{M}_{\text{osc}} (\mathbb{N}^{*})$. Moreover, we have for large $n$:  \\  \\
$\bullet$ For $\alpha \neq p-1$, it follows that: $\displaystyle
w_{\nu}(n) \underset{}{=} \mathcal{O} (n^{\alpha-p}).
$\\  
$\bullet$ For $\alpha = p-1$, it follows that $\displaystyle w_{\nu}(n) \underset{}{=} \mathcal{O} \left( \frac{1}{n \log^p (n)} \right).$
Thus, we recover the results of Corollary \ref{cor 2.4}. For the special case $\alpha = 0$, $w_{\nu}$ is exactly the weight initially derived by \cite{FKP23}, given by:
$$
w_{\nu}(n)= \left(1-\left(1-\frac{1}{n}\right)^{\frac{p-1}{p}}\right)^{p-1}-\left(\left(1+\frac{1}{n}\right)^{\frac{p-1}{p}}-1\right)^{p-1}.
$$
\end{remark}

\section{\label{sec 4} Sharp Discrete Version, Subcriticality, Stability}

In this section, we prove a discrete version of \eqref{eq 1.1}. The proof consists in handling the estimates $B_{\mu,\nu}^{(i)}$ for $ i = 1, 2,$ through a convexity argument, along with the Hermite-Hadamard inequality: for a convex function $f : [a, b] \to \mathbb{R}$, the following holds:
$$ 
f\left(\frac{a+b}{2}\right) \leq \frac{1}{b-a} \int_a^b f(x) \, dx \leq \frac{f(a) + f(b)}{2}.
$$
If $f$ is concave, the inequalities are reversed:
$$
f\left(\frac{a+b}{2}\right) \geq \frac{1}{b-a} \int_a^b f(x) \, dx \geq \frac{f(a) + f(b)}{2}.
$$
The proof follows from elementary techniques, on the one hand, by applying Jensen's inequality with \( f \) and the random variable $\mathcal{U}_{[a,b]}$ uniformly distributed on \([a,b]\), and on the other hand, by using the convexity property. A higher-dimensional version also exists (see \cite{SB18}): if $\Omega \subset \mathbb{R}^n$ is a bounded convex domain and  a convex function $f : \Omega \to \mathbb{R}$:

    $$
    \frac{1}{|\Omega|} \int_\Omega |f(x)| \, dx \leq \frac{c_{n}}{|\partial \Omega|} \int_{\partial \Omega} |f(y)| \, d\sigma(y),
    $$
where $c_n$ is a constant that depends on the dimension.  \\ \\
To derive bounds of the estimates $B_{\mu,\nu}^{(i)}$ for \( i = 1, 2 \), our starting point is the following proposition:
\begin{prop} \label{prop 4.1} For all $\gamma \in \mathbb{R}$, Let us consider, $\displaystyle S_n = \sum_{k=1}^n k^{\gamma - 1},
$ then, for all $n \geq 2$:
\begin{enumerate}
    \item[\textit{ i)}] For all  $1 < \gamma < 2$: $S_{n-1} + S_n < \frac{2 \, n^{\gamma}}{\gamma}$,
    \item[\textit{ ii)}] For all $\gamma > 0$: $S_{n-1} S_n < \frac{n^{2 \gamma}}{\gamma^2}$,
    \item[\textit{iii)}] For all  $\gamma \in \mathbb{R}$, we have $S_n > \exp\left(\frac{\gamma}{n}\right) S_{n-1}$.
\end{enumerate}
\end{prop}

\begin{proof} 
\textit{ i)}: since for $1 < \gamma < 2$ the function $t \to t^{\gamma-1}$ is concave on $(0,\infty)$ it follows that:
$$ S_{n-1}+S_{n} = 1+ \sum_{k=1}^{n-1} \left[ k^{\gamma-1} + (k+1)^{\gamma-1} \right]
\leq 1+ \sum_{k=1}^{n-1} 2 \int_{k}^{k+1} t^{\gamma-1} dt = 1+ 2 \int_{1}^{n} t^{\gamma-1} dt = \frac{\gamma-2}{\gamma} + \frac{2 n^{\gamma}}{\gamma} <  \frac{2 n^{\gamma}}{\gamma}  $$
\textit{ ii)}: starting with the case $1 < \gamma < 2,$ and using AM-GM inequality in addition to \textit{i)} gives us:
$$S_{n-1} S_{n} \leq  \left( \frac{S_{n-1}+S_{n}}{2} \right )^2  <  \frac{n^{2 \gamma}}{\gamma^2}$$
Now for $\gamma \in (0,1] \cup [2,\infty[,$ the function $t \to t^{\gamma-1}$ is convex therefore using Hermite-Hadamard inequality it follows that for all $k \in [|1,n|]:$
$$
k^{\gamma-1} = \left( \frac{k-\frac{1}{2}+  k+\frac{1}{2}}{2} \right)^{\gamma-1} \leq  \int_{k- \frac{1}{2}}^{k+ \frac{1}{2}} t^{\gamma-1} dt $$
By summing up we obtain that:
$$ S_{n} \leq \int_{\frac{1}{2}}^{n+\frac{1}{2}} t^{\gamma-1} dt \leq   \int_{0}^{n+\frac{1}{2}} t^{\gamma-1} dt  = \frac{(n+\frac{1}{2})^{\gamma}}{\gamma} $$
Thus, we may conclude that:
$ \displaystyle S_{n-1} S_{n} \leq  \frac{(n+\frac{1}{2} )^{\gamma} (n-\frac{1}{2})^{\gamma}}{\gamma^2} = \frac{(n^2-\frac{1}{4})^{\gamma}}{\gamma^2} < \frac{n^{2 \gamma}}{\gamma^2}$

\textit{(iii):}
 To address the last point, we analyze the cases based on the values of $\gamma:$ \\  \\
\underline{if $\gamma \leq 0,$} this is a trivial case since the inequality follows directly from \( S_{n} > S_{n-1} \). \\ \\ 
\underline{if $\gamma \in (0,1]\cup (1,2)$}: We start by establishing the following result: for all $x>0:$
$$ \frac{e^{x}}{e^{x}-1} - \frac{1}{2} - \frac{1}{x} > 0 $$

Define $\displaystyle f:x \to \frac{e^x}{e^x-1} - \frac{1}{2} - \frac{1}{x}.$ By analyzing its derivative, we obtain for all $x\in (0,\infty):$
$$ f'(x)= \frac{e^{x} (e^{x}+e^{-x}-x^2-2)}{x^2 (e^{x} -1)^2  }, $$ 
 
therefore, $f'(x)>0$ since $e^{x} + e^{-x} - x^2 - 2 = \displaystyle 2 \sum_{k=2}^{\infty} \frac{x^{2k}}{(2k)!}> 0,$ \\ \\
hence for all $x\in (0,\infty):$ \ 
$\displaystyle f(x) > \lim\limits_{t\to 0^{+}} f(t) = \lim\limits_{t\to 0^{+}} \frac{t}{12} + o(t) = 0.$ Now for $\gamma \in (0,1],$ since $t\to t^{\gamma-1}$  is decreasing, we derive:
$$ \frac{S_n}{n^{\gamma-1}} = \frac{1}{n^{\gamma-1}}  \sum_{k=1}^{n} k^{\gamma-1} \leq  \frac{1}{n^{\gamma-1}}  \sum_{k=1}^{n} \int_{k-1}^{k} t^{\gamma-1} \ dt =  \frac{1}{n^{\gamma-1}} \int_{0}^{n} t^{\gamma-1} \ dt = \frac{n}{\gamma},
$$
Thus, we conclude that:
$$
\frac{S_n}{n^{\gamma-1}} \leq  \frac{n}{\gamma} < \left(\frac{\exp(\frac{\gamma}{n})}{\exp(\frac{\gamma}{n})-1}\right) -\frac{1}{2} < \frac{\exp(\frac{\gamma}{n})}{\exp(\frac{\gamma}{n})-1}.$$
which is equivalent to
$$ S_{n} > \exp(\frac{\gamma}{n})(S_{n}-n^{\gamma-1}) = \exp(\frac{\gamma}{n}) S_{n-1}.$$
For the case $\gamma \in (1,2)$ applying the result of \textit{ i)} along with the above inequality, we obtain:
$$ \frac{S_n}{n^{\gamma-1}} = \frac{1}{2 n^{\gamma-1}} (S_{n}+S_{n-1}) +\frac{1}{2} < \frac{n}{\gamma} + \frac{1}{2} < \frac{\exp(\frac{\gamma}{n})}{\exp(\frac{\gamma}{n})-1}.
$$
which is also equivalent to $S_{n} > \exp(\frac{\gamma}{n}) S_{n-1}.$ \\ \\
\underline{if $\gamma \geq 2$}: 
We will show by induction on \( n \) that for all $n \geq 1:$
$$ \frac{S_n}{n^{\gamma-1}} < \frac{\exp(\frac{\gamma}{n})}{\exp(\frac{\gamma}{n})-1}
$$
To complete the proof, we rely on the following lemma, whose proof is deferred to the end of the section.
\begin{lem} \label{lemma 4.2} Consider the function 
$F : [0,1]\times[2,\infty[ \longrightarrow  \R$ defined as follows:
$$\begin{array}{lrcl}
     &(x,\gamma) & \longmapsto & (1+x)^{\gamma-1} (1-exp(-\gamma x) ) - \exp \left( \frac{\gamma x}{1+x} \right) +1 \end{array}
$$
then for all $n \geq 1$ and $\gamma \geq 2,$ we have $F(\frac{1}{n} ,\gamma) \geq 0.$
\end{lem}

We proceed to prove the statement using induction.
For \( n = 1 \), the inequality holds trivially. Assuming that the inequality holds for \( n \), we aim to establish it for \( n + 1 \). First, by the induction hypothesis, we have:
$$ \frac{S_{n+1}} {(n+1)^{\gamma-1}} = ( 1+ \frac{1}{n})^{1-\gamma} \,  \frac{S_{n}}{n^{\gamma-1}} +1 <
( 1+ \frac{1}{n})^{1-\gamma} \, \left[ \frac{\exp(\frac{\gamma}{n})}{\exp(\frac{\gamma}{n})-1} \right] + 1$$
From Lemma \ref{lemma 4.2}, we have the following equivalent formulation
$$( 1+ \frac{1}{n})^{1-\gamma}  \, \left[ \frac{\exp(\frac{\gamma}{n})}{\exp(\frac{\gamma}{n})-1} \right]  \leq \frac{1}{\exp(\frac{\gamma}{n+1})-1}$$
By combining the induction hypothesis with the result of Lemma \ref{lemma 4.2}, we derive:
$$\frac{S_{n+1}} {(n+1)^{\gamma-1}}
<   \frac{1}{\exp(\frac{\gamma}{n+1})-1} +1  =  \frac{\exp(\frac{\gamma}{n+1})}{\exp(\frac{\gamma}{n+1})-1}.$$ Hence, the induction step is verified, and the inequality holds for all $n \geq 1:$ (with the convention $S_{0}=0$).
$$ \frac{S_n}{n^{\gamma-1}} < \frac{\exp(\frac{\gamma}{n})}{\exp(\frac{\gamma}{n})-1} \Leftrightarrow S_{n} > \exp(\frac{\gamma}{n}) S_{n-1}.
$$
This concludes the proof of part \textit{iii)} and completes the proposition.
\end{proof}

Using Proposition \ref{prop 4.1}, we derive the following corollary:
\begin{cor} \label{cor 4.3} For a real $\alpha > p-1,$ Let $\displaystyle S_{n} = \sum_{k=1}^{n} k^{\alpha-p}$ then for all $n\geq 2$:
$$ \left( \frac{1}{S_{n-1}} \right)^{\frac{1}{p-1}}- \left( \frac{1}{S_{n}} \right)^{\frac{1}{p-1}} \geq \frac{(\alpha-p+1)^{\frac{p}{p-1}}}{(p-1) n^{\frac{\alpha}{p-1}}} 
\sum_{k=0}^{\infty} \frac{(\alpha-p+1)^{2k}}{2^{2k}(p-1)^{2k} (2k+1)!  \ n^{2k}} $$ \\
In particular, we have:
$$
 \left( \frac{1}{S_{n-1}} \right)^{\frac{1}{p-1}}- \left( \frac{1}{S_{n}} \right)^{\frac{1}{p-1}} \geq  \frac{(\alpha-p+1)^{\frac{p}{p-1}}}{(p-1) n^{\frac{\alpha}{p-1}}} 
$$
\end{cor}
\begin{proof} To begin, we rewrite the left-hand side as follows:
$$
\left( \frac{1}{S_{n-1}} \right)^{\frac{1}{p-1}}- \left( \frac{1}{S_{n}} \right)^{\frac{1}{p-1}} = \left( \frac{1}{S_{n-1} S_{n}} \right)^{\frac{1}{2(p-1)}} \left[\left( \frac{S_{n}}{S_{n-1}} \right)^{\frac{1}{2(p-1)}}- \left( \frac{S_{n}}{S_{n-1}} \right)^{\frac{-1}{2(p-1)}}\right]
$$
Now, since: $x \to x^{\frac{1}{2(p-1)}} - x^{\frac{-1}{2(p-1)}}$ is non decreasing on $(0,\infty)$ and $S_{n} \geq  \exp(\frac{\alpha-p+1}{n}) S_{n-1}$
$$
\geq  \left( \frac{1}{S_{n-1} S_{n}} \right)^{\frac{1}{2(p-1)}} \left[ \exp\left( \frac{\alpha-p+1} {2(p-1)n} \right)- \exp\left( \frac{-\alpha+p-1} {2(p-1)n} \right) \right]
$$
$$
 > \left(\frac{\alpha -p+1}{n^{\alpha-p+1}} \right)^{\frac{1}{p-1}} \left[ \exp\left( \frac{\alpha-p+1} {2(p-1)n} \right)- \exp\left( \frac{-\alpha+p-1} {2(p-1)n} \right) \right]
$$
Using the series expansion of the exponential function, we obtain:
$$
=  \left(\frac{\alpha -p+1}{n^{\alpha-p+1}} \right)^{\frac{1}{p-1}}  \sum_{k=0}^{\infty} \frac{(\alpha-p+1)^{2k+1}}{2^{2k}(p-1)^{2k+1} (2k+1)!  \ n^{2k+1}}
$$

$$ = \frac{(\alpha-p+1)^{\frac{p}{p-1}}}{(p-1) n^{\frac{\alpha}{p-1}}} 
\sum_{k=0}^{\infty} \frac{(\alpha-p+1)^{2k}}{2^{2k}(p-1)^{2k} (2k+1)!  \ n^{2k}}
$$

which lead us to the desired result.
\end{proof}

\begin{remark} \label{rmk 4.4} The third property was crucial in obtaining infinite terms in the sum. In fact, the second inequality, for the cases  $p\in ]1,2]$ or $p>2$ and $\alpha \in (p-1,p+1)$, can be shown directly as follows:  
$$\left( \frac{1}{S_{n-1}} \right)^{\frac{1}{p-1}}- \left( \frac{1}{S_{n}} \right)^{\frac{1}{p-1}} \geq  \frac{(\alpha-p+1)^{\frac{p}{p-1}}}{(p-1) n^{\frac{\alpha}{p-1}}}$$ \\ 
For $p \in ]1,2]:$ by applying the Hermite-Hadamard inequality to the function $t \to t^{\frac{2-p}{p-1}}$, we derive the following:

\begin{align*}
\left( \frac{1}{S_{n-1}} \right)^{\frac{1}{p-1}}
-
\left( \frac{1}{S_n} \right)^{\frac{1}{p-1}}
&=
\frac{1}{p-1}
\int_{\frac{1}{S_n}}^{\frac{1}{S_{n-1}}}
t^{\frac{2-p}{p-1}} \, dt  \notag \\
&\ge
\begin{cases}
\displaystyle
\frac{1}{2(p-1)}
\left( \frac{1}{S_{n-1}} - \frac{1}{S_n} \right)
\left[
\left( \frac{1}{S_{n-1}} \right)^{\frac{2-p}{p-1}}
+
\left( \frac{1}{S_n} \right)^{\frac{2-p}{p-1}}
\right],
& \text{si } p \in \left[\frac{3}{2},2\right],
\\[1.2em]
\displaystyle
\frac{1}{p-1}
\left( \frac{1}{S_{n-1}} - \frac{1}{S_n} \right)
\left[
\frac{1}{2S_{n-1}}+\frac{1}{2S_n}
\right]^{\frac{2-p}{p-1}},
& \text{si } p \in \left(1,\frac{3}{2}\right).
\end{cases}
\end{align*}
In both cases, by applying the AM-GM inequality, we obtain the following:
$$
\left( \frac{1}{S_{n-1}} \right)^{\frac{1}{p-1}} - \left( \frac{1}{S_{n}} \right)^{\frac{1}{p-1}} \geq \frac{1}{p-1} \left( \frac{1}{S_{n-1}} - \frac{1}{S_n} \right) 
\left[ \frac{1}{S_{n-1}S_{n}} \right]^{ \frac{2-p}{2(p-1)} }
= \frac{n^{\alpha-p}}{p-1} \left[ \frac{1}{S_{n-1}S_{n}} \right]^{ \frac{p}{2(p-1)} } \geq \frac{(\alpha-p+1)^{\frac{p}{p-1}}}{(p-1) n^{\frac{\alpha}{p-1}}}.
$$
Here, the final inequality leverages the second point of the previous proposition. \\ \\
For $p>2$ and $\alpha \in (p-1,p+1)$
, using the Hermite-Hadamard inequality, we derive:
$$
\left( \frac{1}{S_{n-1}} \right)^{\frac{1}{p-1}}- \left( \frac{1}{S_{n}} \right)^{\frac{1}{p-1}}=\frac{1}{p-1}
\int_{\frac{1}{S_{n}}}^{\frac{1}{S_{n-1}}} t^{\frac{2-p}{p-1}} dt \geq
\frac{1}{p-1} \left( \frac{1}{S_{n-1}} - \frac{1}{S_n} \right) \left[ \frac{1}{2 S_{n-1}}+ \frac{1}{2 S_{n}} \right]^{\frac{2-p}{p-1}} $$

$$= \frac{n^{\alpha-p}}{p-1} \left( \frac{1}{S_{n-1} S_{n}} \right)^{\frac{1}{p-1}} \left[ \frac{S_{n-1}+S_{n}}{2}\right]^{\frac{2-p}{p-1}} $$ 

For $\alpha \in (p-1,p),$ it follows that:
 $\frac{S_{n-1}+S_{n}}{2} <  S_{n}
\leq \frac{n^{\alpha-p+1}}{\alpha-p+1}$
Moreover for $\alpha \in (p,p+1)$ by using \textit{(i)} in Proposition \ref{prop 4.1} we have : $\frac{S_{n-1}+S_{n}}{2}
< \frac{n^{\alpha-p+1}}{\alpha-p+1}.$  Applying \textit{(ii)} from Proposition \ref{prop 4.1}, we deduce:
$$
\frac{n^{\alpha-p}}{p-1} \left( \frac{1}{S_{n-1} S_{n}} \right)^{\frac{1}{p-1}} \left[ \frac{S_{n-1}+S_{n}}{2}\right]^{\frac{2-p}{p-1}} 
\geq \frac{(\alpha-p+1)^{\frac{p}{p-1}}}{(p-1) n^{\frac{\alpha}{p-1}}} $$
It is important to note that this last inequality does not hold for $\alpha > p+1,$ highlighting the critical role of the exponential growth property of $S_{n}.$
\end{remark}

The following result, proved in \cite{Liu12}, provides an improvement of the discrete version of \eqref{eq 1.1} for $-1<\alpha<p-1$.
\begin{thm} \label{thm 4.5}
Let $p>1$ and $\alpha\in\mathbb{R}$ such that $-1<\alpha<p-1$. Consider a sequence $(a_n)_{n\ge1}$ of nonnegative real numbers satisfying $\displaystyle\sum_{n=1}^{\infty}\left(n-\tfrac12\right)^{\alpha}a_n<\infty.$
Then
\[
\sum_{n=1}^{\infty} n^{\alpha}
\left(
\frac{\displaystyle\sum_{k=1}^{n} a_k^{1/p}}{n}
\right)^p
\le
\left(\frac{p}{p-1-\alpha}\right)^p
\sum_{n=1}^{\infty}
\left(n-\tfrac12\right)^{\alpha} a_n .
\]
\end{thm}
We now establish a discrete version of \ref{eq 1.1}, valid for all non-negative \( \alpha \) such that \( \alpha \neq p - 1 \),
\begin{thm} \label{thm 4.6} Let $\alpha$ a non-negative real such that: $\alpha \neq p-1.$ Then for all $u \in H_0^{1}(\mathbb{N}):$
$$\sum_{n=1}^{\infty} |u_{n}-u_{n-1}|^{p} \ n^{\alpha}  \geq  \mathcal{A}^{\text{disc}}_{\alpha,p} \, \sum_{n=1}^{\infty} |u_{n}|^{p} \ n^{\alpha-p}, $$
with $\mathcal{A}^{\text{disc}}_{\alpha,p}  = \left|\frac{\alpha-p+1}{p} \right|^{p} = \mathcal{A}^{\text{cont}}_{\alpha,p} ,$ Furthermore, as in the continuum settings, $\mathcal{A}^{\text{disc}}_{\alpha,p} $ is sharp.
\end{thm}

\begin{proof} \underline{if $\alpha > p-1$}: Set for every $n\geq 1: f(n)= u_{n}-u_{n-1}$ and $\mu(n)= n^{\alpha-p},$ $\nu(n)= n^{\alpha}$ 
then $f \in F^{0}(\mathbb{N}^{*})$ therefore we can use the minima bound from the Proposition \ref{prop 2.3}:
$$\min(B_{\mu,\nu}^{(1)},B_{\mu,\nu}^{(2)}) = \sup_{r \geq 1} \left[ \left( \sum_{k=1}^{r} k^{\alpha-p} \right) \cdot \left( \sum_{k \geq r+1} \frac{1}{{k}^{\frac{\alpha}{p-1}}} \right)^{p-1} \right]$$
Let us denote $\displaystyle S_{r} =\sum_{k=1}^{r} k^{\alpha-p}$ using Corollary \ref{cor 4.3} then for every $k\geq r+1:$
$$
 \left( \frac{1}{S_{k-1}} \right)^{\frac{1}{p-1}}- \left( \frac{1}{S_{k}} \right)^{\frac{1}{p-1}} \geq  \frac{(\alpha-p+1)^{\frac{p}{p-1}}}{(p-1) k^{\frac{\alpha}{p-1}}} 
$$
Summing up from $r+1$ to $\infty$ and using the fact  that $\lim\limits_{n \to \infty} \frac{1}{S_{n}} = 0,$ therefore:
$$ \left( \frac{1}{S_{r}} \right)^{\frac{1}{p-1}}  \geq  \frac{(\alpha-p+1)^{\frac{p}{p-1}}}{(p-1)}  \sum_{k=r+1}^{\infty} \frac{1}{k^{\frac{\alpha}{p-1}}}
$$
Hence for all $r\geq 1:$
$$\left( \sum_{k=1}^{r} k^{\alpha-p} \right) \cdot \left( \sum_{k \geq r+1} \frac{1}{{k}^{\frac{\alpha}{p-1}}} \right)^{p-1} \leq \frac{(p-1)^{p-1}} {(\alpha-p+1)^p}
$$
As a consequence $\min(B_{\mu,\nu}^{(1)},B_{\mu,\nu}^{(2)})  \leq \frac{(p-1)^{p-1}} {(\alpha-p+1)^p}$ and therefore:
$$
\sum_{n=1}^{\infty} |u_{n}|^{p} \ n^{\alpha-p} \leq \frac{p^{p}}{(\alpha-p+1)^p}  \ \sum_{n=1}^{\infty} |u_{n}-u_{n-1}|^{p} \ n^{\alpha}
$$
This yields the equivalent form:
$$
\sum_{n=1}^{\infty} |u_{n}-u_{n-1}|^{p} \ n^{\alpha}  \geq  \left(\frac{\alpha-p+1}{p} \right)^p \ \sum_{n=1}^{\infty} |u_{n}|^{p} \ n^{\alpha-p} = \mathcal{A}^{\text{disc}}_{\alpha,p} \sum_{n=1}^{\infty} |u_{n}|^{p} \ n^{\alpha-p} $$ 
\underline{if $ 0 \leq \alpha < p-1$}: Choosing $a_n=|u_n-u_{n-1}|^p , \, \,n\geq 1$ in Theorem  \ref{thm 4.5}, together with the inequality $\left(n-\tfrac12\right)^{\alpha}\le n^{\alpha},$ allows us to conclude. However, the result can also be obtained by a different argument for $u\in H_0^2(\mathbb{N})$. In this approach, the weight that appears is even smaller, namely $(n-1)^{\alpha}$ instead of $\left(n-\tfrac12\right)^{\alpha}$. \\ \\
Take for every $n\geq 1: f(n)= u_{n+1}-u_{n}$ and $\mu(n)= (n+1)^{\alpha-p},$ $\nu(n)= n^{\alpha}$ hence for every $n\geq 1:$ $\displaystyle \sum_{k=1}^{n} f(k) = u_{n+1}$  since $B_{\mu,\nu}^{(1)} < \infty,$ it follows that:
$$
\sum_{n=1}^{\infty} |u_{n+1}|^{p} \ (n+1)^{\alpha-p} \leq \frac{p^{p} \ B_{\mu,\nu}^{(1)}}{(p-1)^{p-1}}  \ \sum_{n=1}^{\infty} |u_{n+1}-u_{n}|^{p} \ n^{\alpha}  $$
where $ \displaystyle B_{\mu,\nu}^{(1)} = \sup_{r \geq 1} \left[ \left( \sum_{k=r}^{\infty} (k+1)^{\alpha-p} \right) \cdot \left( \sum_{k=1}^{r} \frac{1}{k^{\frac{\alpha}{p-1}}} \right)^{p-1} \right]$ Consistently, through the fact that:
$$\sum_{k=1}^{r} \frac{1}{k^{\frac{\alpha}{p-1}}} \leq \sum_{k=1}^{r} \int_{k-1}^{k} t^{\frac{-\alpha}{p-1}} \, dt =  \int_{0}^{r} t^{\frac{-\alpha}{p-1}} \, dt = \frac{p-1}{p-1-\alpha} \, r^{\frac{p-1-\alpha}{p-1}}
$$
and 
$$\sum_{k=r}^{\infty} (k+1)^{\alpha-p}
=\sum_{k=r+1}^{\infty} k^{\alpha-p}
\leq \sum_{k=r+1}^{\infty} \int_{k-1}^{k} t^{\alpha-p}\,dt = \int_{r}^{\infty} t^{\alpha-p}\,dt
=\frac{r^{\alpha-p+1}}{p-1-\alpha}
$$
As a result, it naturally follows that: $B_{\mu,\nu}^{(1)} \leq \frac{(p-1)^{p-1}} {(p-1-\alpha)^p},$ then:
$$
\sum_{n=1}^{\infty} |u_{n+1}|^{p} \ (n+1)^{\alpha-p} \leq \frac{p^{p}}{(p-1-\alpha)^p}  \ \sum_{n=1}^{\infty} |u_{n+1}-u_{n}|^{p} \ n^{\alpha}. 
$$
This leads to the equivalent form of the inequality as follows: \\
$$\sum_{n=2}^{\infty} |u_{n}|^{p} \ n^{\alpha-p} \leq 
\frac{p^{p}}{(p-1-\alpha)^p}  \ \sum_{n=2}^{\infty} |u_{n}-u_{n-1}|^{p} \ (n-1)^{\alpha} \leq 
\frac{p^{p}}{(p-1-\alpha)^p}  \ \sum_{n=2}^{\infty} |u_{n}-u_{n-1}|^{p} \ n^{\alpha}
$$
\phantom{.}\\ 
Next, we establish the sharpness of the constant $\mathcal{A}^{\text{disc}}_{\alpha,p} $ is optimal. Following the approach of  \cite{Ye23}, we take an integer $m \geq 1$ and $\phi \in C^{\infty}_{0} (0,1)$ therefore if we set for all $n\geq 0:$ $u_{n}= \phi(\frac{n}{m}),$ we obtain:
$$\sum_{n=1}^{\infty} |u_{n}-u_{n-1}|^{p} \ n^{\alpha} = \sum_{n=1}^{m} \left|\phi\left(\frac{n}{m}\right)-\phi\left(\frac{n-1}{m}\right) \right|^{p} \ n^{\alpha} 
$$
By Taylor-Lagrange formula, it follows that for every $1\leq n \leq m$:
$$\phi\left(\frac{n}{m}\right)-\phi\left(\frac{n-1}{m}\right) = \frac{1}{m} \phi'\left(\frac{n}{m}\right) + r_{n} $$
such that $|r_n| \leq \frac{ \left\|\phi'' \right\|_{\infty}}{2m^2},$ therefore it follows that:

$$\sum_{n=1}^{m} \left|\phi\left(\frac{n}{m}\right)-\phi\left(\frac{n-1}{m}\right) \right|^{p} \ n^{\alpha} = m^{\alpha-p} \sum_{n=1}^{m} \left|\phi'\left(\frac{n}{m}\right) \right|^{p} \ \left(\frac{n}{m}\right)^{\alpha} + \mathcal{O} (m^{\alpha-p})
$$
And similarly, we have:
$$
\sum_{n=1}^{\infty} |u_{n}|^{p} \ n^{\alpha-p} =  \sum_{n=1}^{\infty} \left| \phi \left( \frac{n}{m} \right) \right|^p \ n^{\alpha-p}
= m^{\alpha-p} \sum_{n=1}^{m} \left| \phi \left( \frac{n}{m} \right) \right|^p \ \left(\frac{n}{m}\right)^{\alpha-p}
$$
therefore:
$$
\displaystyle \inf_{v \in H_{0}^{1}(\mathbb{N})} \frac{\displaystyle \sum_{n=1}^{\infty} |v_{n}-v_{n-1}|^{p} \ n^{\alpha} } {\displaystyle \sum_{n=1}^{\infty} |v_{n}|^{p} \ n^{\alpha-p} } \leq   \frac{\displaystyle  \frac{1}{m} \sum_{n=1}^{m} \left|\phi'\left(\frac{n}{m}\right) \right|^{p} \ \left(\frac{n}{m}\right)^{\alpha} + \mathcal{O} (m^{-1})}{\displaystyle \frac{1}{m} \sum_{n=1}^{m} \left| \phi \left( \frac{n}{m} \right) \right|^p \ \left(\frac{n}{m}\right)^{\alpha-p}} \underset{m \to +\infty}{\longrightarrow} \frac{\displaystyle \int_{0}^{1} |\phi'(x)|^p x^{\alpha} dx}{\displaystyle \int_{0}^{1} |\phi(x)|^p x^{\alpha-p} dx}
$$
As $\phi$ was chosen arbitrarily, it follows that:
$$\displaystyle \inf_{v \in H_{0}^{1}(\mathbb{N})} \frac{\displaystyle \sum_{n=1}^{\infty} |v_{n}-v_{n-1}|^{p} \ n^{\alpha} } {\displaystyle \sum_{n=1}^{\infty} |v_{n}|^{p} \ n^{\alpha-p} } \leq  \inf_{\phi \in C^{\infty}_{0} (0,1)} \frac{\displaystyle \int_{0}^{1} |\phi'(x)|^p x^{\alpha} dx}{\displaystyle \int_{0}^{1} |\phi(x)|^p x^{\alpha-p} dx}
$$
On the other hand, using a scaling argument, we have:
$$
 \inf_{\phi \in C^{\infty}_{0} (0,1)} \frac{\displaystyle \int_{0}^{1} |\phi'(x)|^p x^{\alpha} dx}{\displaystyle \int_{0}^{1} |\phi(x)|^p x^{\alpha-p} dx} = 
  \inf_{\phi \in C^{\infty}_{0} (0,\infty)} \frac{\displaystyle \int_{0}^{\infty} |\phi'(x)|^p x^{\alpha} dx}{\displaystyle \int_{0}^{\infty} |\phi(x)|^p  x^{\alpha-p} dx} = \mathcal{A}^{\text{cont}}_{\alpha,p} = \left|\frac{\alpha-p+1}{p} \right|^p = \mathcal{A}^{\text{disc}}_{\alpha,p}$$
Thus, the sharpness is established, and this completes the proof.
\end{proof}

\begin{remark} \label{rmk 4.7} For $\alpha \in (p+1,\infty)$, through Hermite-Hadamard inequality, we have:
$$
\sum_{k=1}^{r} k^{\alpha-p} 
\leq \sum_{k=1}^{r} \int_{k-\frac{1}{2}}^{k+\frac{1}{2}} t^{\alpha-p} \,dt \leq   \int_{0}^{r+\frac{1}{2}} t^{\alpha-p}\,dt  =  \frac{1} {\alpha-p+1}  \left(r+\frac{1}{2}\right)^{\alpha-p+1}$$
Similarly, applying the same approach, we obtain the following:
$$ \sum_{k=r+1}^{\infty} \frac{1}{k^{\frac{\alpha}{p-1}}} \leq  \sum_{k=r+1}^{\infty} \int_{k-\frac{1}{2}}^{k+\frac{1}{2}} \frac{dt}{t^{\frac{\alpha}{p-1}}}\, = \int_{r+ \frac{1}{2}}^{\infty} \ \frac{dt}{t^{\frac{\alpha}{p-1}}} = \frac{p-1}{\alpha-p+1} \left(r+\frac{1}{2}\right)^{\frac{p-1-\alpha}{p-1}} 
$$
and therefore $\min(B_{\mu,\nu}^{(1)},B_{\mu,\nu}^{(2)})  \leq \frac{(p-1)^{p-1}} {(\alpha-p+1)^p}.$ This result recovers the inequality for  \( \alpha > p - 1 \), as outlined in Remark \ref{rmk 4.4}. However, the utility of Corollary \ref{cor 4.3}, with an infinite number of terms, provides deeper insight into the functional energy associated with the inequality, as well as its stability.
\end{remark} 

We also have a discrete version of \ref{eq 1.2} for the critical case \( \alpha = p - 1 \):
\begin{prop} \label{prop 4.8} For all $u \in H_0^{2}(\mathbb{N}),$ the following inequality holds: 
$$\sum_{n=2}^{\infty} |u_{n}-u_{n-1}|^{p} \ n^{p-1}  \geq  \left( \frac{p-1}{p}\right)^{p} \, \sum_{n=2}^{\infty} \frac{|u_{n}|^{p}} {n \log^{p}(n)} $$
\end{prop}
\begin{proof}
Set $\displaystyle f(n)=u_{n+1}-u_n$ and
$\displaystyle\mu(n)=
\frac{\log(n+1)- \log(n)}
{\log^p(n+1)},
\quad
\displaystyle \nu(n)=
\left(
\log(n+1)-\log(n)
\right)^{1-p}.$ \\
Since
\[
\sum_{k=1}^{r}
\frac{1}{\nu(k)^{\frac{1}{p-1}}}
=
\log(r+1)
\]
\phantom{.}\\
Moreover, since $t\mapsto t^{-p}$ is decreasing, \\
\[
\sum_{n=r+1}^{\infty}\mu(n)
=
\sum_{n=r+1}^{\infty}
\frac{\log(n+1)-\log (n)}{\log^p(n+1)}
\leq
\sum_{n=r+1}^{\infty}
\int_{\log (n)}^{\log(n+1)}t^{-p}\,dt=
\frac{1}{(p-1)\log^{p-1}(r+1)}.
\]
Consequently,
\[
B_{\mu,\nu}^{(1)} \leq
\sup_{r\geq1}
\left[
\frac{\log(r+1)-\log(r)}
{\log(r+1)}
+\frac{1}{p-1}
\right] = \frac{p}{p-1}.
\]
we obtain
\[
B_{\mu,\nu}^{(1)}
\leq
\frac{p}{p-1}<\infty.
\]
Thus, Proposition \ref{Prop 2.1}
applies. To obtain the desired constant, we refine the estimate using convexity. For every $n\geq1$, define
\[
h(n)=
\frac{\sum_{k=1}^{n}|f(k)|}
{\log(n+1)},
\qquad h(0)=0.
\]
Then
\[
\log(n+1)h(n)
-
\log(n)h(n-1)
=
|f(n)|.
\] 
By the convexity of $x\mapsto x^p$, we have
\[
h(n)^p-h(n-1)^p
\leq
p h(n)^{p-1}(h(n)-h(n-1)).
\]
Moreover, the definition of $h$ yields
\[
\log(n)(h(n)-h(n-1))
=
|f(n)|
-
\log\left(\frac{n+1}{n}\right)h(n).
\]
Combining these identities, we obtain
\[ \log(n)\left(h(n)^p-h(n-1)^p\right) \leq
p|f(n)|h(n)^{p-1}
-
p\log\left(\frac{n+1}{n}\right)h(n)^p.
\]
Summing from $n=1$ to $N,$ we obtain
\[\sum_{n=1}^{N}
\log(n)\left(h(n)^p-h(n-1)^p\right) =
\log(N+1)h(N)^p
-
\sum_{n=1}^{N}
\log\left(\frac{n+1}{n}\right)h(n)^p,
\]
we deduce that
\[(p-1)\sum_{n=1}^{N}
\log\left(\frac{n+1}{n}\right)h(n)^p
+
\log(N+1)h(N)^p \leq
p\sum_{n=1}^{N}|f(n)|h(n)^{p-1}.
\]
It follows that
\[(p-1)\sum_{n=1}^{N}
\log\left(\frac{n+1}{n}\right)h(n)^p
\leq
p\sum_{n=1}^{N}|f(n)|h(n)^{p-1}.
\]
Now, applying H\"older's inequality, we obtain
\[
\begin{aligned}
\sum_{n=1}^{N}|f(n)|h(n)^{p-1}
&=
\sum_{n=1}^{N}
\left(\nu(n)^{1/p}|f(n)|\right)
\left(\nu(n)^{-1/p}h(n)^{p-1}\right)\\
&\leq
\left(
\sum_{n=1}^{N}\nu(n)|f(n)|^p
\right)^{1/p}
\left(
\sum_{n=1}^{N}
\nu(n)^{-\frac{1}{p-1}}h(n)^p
\right)^{\frac{p-1}{p}}.
\end{aligned}
\]
Since
$\nu(n)^{-\frac{1}{p-1}}
=\log\left(\frac{n+1}{n}\right),$
we conclude that
\[(p-1)\sum_{n=1}^{N}
\log\left(\frac{n+1}{n}\right)h(n)^p \leq
p\left(
\sum_{n=1}^{N}\nu(n)|f(n)|^p
\right)^{1/p}\
\left(
\sum_{n=1}^{N}
\log\left(\frac{n+1}{n}\right)h(n)^p
\right)^{\frac{p-1}{p}}.
\]
Dividing by the last factor, and raising to the power $p$, we finally obtain
\[
\sum_{n=1}^{N}
\log\left(\frac{n+1}{n}\right)h(n)^p
\leq
\left(\frac{p}{p-1}\right)^p
\sum_{n=1}^{N}\nu(n)|f(n)|^p.\]
Since $u_1=0$, we have
$|u_{n+1}|
\leq
\sum_{k=1}^{n}|f(k)|.$ Moreover,
$
\log\left(\frac{n+1}{n}\right)
\geq
\frac{1}{n+1},$
and consequently \\
\[
\nu(n)
\leq
(n+1)^{p-1}.
\]
Letting $N\to\infty$, we conclude that
\[
\begin{aligned}
\sum_{n=2}^{\infty}
\frac{|u_n|^p}{n\log^p n}
&\leq
\sum_{n=1}^{\infty}
\log\left(\frac{n+1}{n}\right)h(n)^p\\
&\leq
\left(\frac{p}{p-1}\right)^p
\sum_{n=1}^{\infty}
\nu(n)|f(n)|^p\\
&\leq
\left(\frac{p}{p-1}\right)^p
\sum_{n=2}^{\infty}
n^{p-1}|u_n-u_{n-1}|^p.
\end{aligned}
\]
This completes the proof.
\end{proof}

We now examine the case of $\alpha < 0.$ According to Corollary \ref{cor 2.4}, the following inequality holds for all $u \in H_{0}^{1}(\mathbb{N}):$
\begin{equation}
\sum_{n=1}^{\infty} |u_{n}-u_{n-1}|^{p} \ n^{\alpha}  \geq  \mathcal{A}^{\text{disc}}_{\alpha,p} \, \sum_{n=1}^{\infty} |u_{n}|^{p} \ n^{\alpha-p}, \tag{4.7} \label{4.7}
\end{equation}
where $\mathcal{A}^{\text{disc}}_{\alpha,p} $ is the sharpest constant for which the inequality holds. However, the majorization $(n- \frac{1}{2})^{\alpha} \leq n^{\alpha},$ used in the proof of Theorem \ref{thm 4.6}, does not hold for $\alpha<0$. Nonetheless, we establish a bounded control of $\mathcal{A}_{\alpha,p}^{\text{disc}}$ in terms of its continuous counterpart $\mathcal{A}_{\alpha,p}^{\text{cont}}$.

\begin{prop} \label{prop 4.9} For $\alpha < 0$, the following bounds hold true: 
$$
2^{\alpha-p} \mathcal{A}_{\alpha,p}^{\text{cont}} \leq \mathcal{A}_{\alpha,p}^{\text{disc}} \leq \mathcal{A}_{\alpha,p}^{\text{cont}}.
$$
\end{prop} 
\begin{proof}
Let $ f(n)= u_{n}-u_{n-1},$  $\mu(n)= (n+1)^{\alpha-p},$ and $\nu(n)= n^{\alpha}$ for every $n\geq 1.$ \\
Since $B_{\mu,\nu}^{(1)} < \infty,$  we deduce that:
$$\sum_{n=1}^{\infty} |u_{n}|^{p} \ (n+1)^{\alpha-p} \leq \frac{p^{p} \ B_{\mu,\nu}^{(1)}}{(p-1)^{p-1}}  \ \sum_{n=1}^{\infty} |u_{n}-u_{n-1}|^{p} \ n^{\alpha}$$
where $$\displaystyle B_{\mu,\nu}^{(1)} = \sup_{r \geq 1} \left[ \left( \sum_{k=r}^{\infty} (k+1)^{\alpha-p} \right) \cdot \left( \sum_{k=1}^{r} \frac{1}{k^{\frac{\alpha}{p-1}}} \right)^{p-1} \right].$$  
Set $\beta=-\frac{\alpha}{p-1}>0.$
By the Hermite--Hadamard inequality,
for every $r\geq1$, we have
\[
\sum_{k=1}^{r}k^\beta
\leq
\int_0^{r+\frac12}t^\beta\,dt
=
\frac{(r+\frac12)^{\beta+1}}{\beta+1}.
\]
Indeed, for $\beta\geq1$, this follows
from the midpoint inequality for convex
functions. For $0<\beta<1$, the
inequality for concave
functions gives
\[
\sum_{k=1}^{r}k^\beta
\leq
\int_0^r t^\beta\,dt
+\frac{r^\beta}{2}
\leq
\int_0^{r+\frac12}t^\beta\,dt.
\]
Moreover, since $t\mapsto t^{\alpha-p}$
is convex, Hermite--Hadamard yields
\[
\sum_{k=r}^{\infty}(k+1)^{\alpha-p}
=
\sum_{k=r+1}^{\infty}k^{\alpha-p}
\leq
\int_{r+\frac12}^{\infty}
t^{\alpha-p}\,dt
=
\frac{(r+\frac12)^{\alpha-p+1}}
{p-1-\alpha}.
\]
Consequently, using $(p-1)(\beta+1)=p-1-\alpha,$ we obtain
\[
B_{\mu,\nu}^{(1)}
\leq
\frac{1}
{(p-1-\alpha)(\beta+1)^{p-1}}=
\frac{(p-1)^{p-1}}
{(p-1-\alpha)^p}.
\]
As a result, we obtain:

$$ \sum_{n=1}^{\infty} |u_{n}-u_{n-1}|^{p} \ n^{\alpha} \geq  \frac {(p-1-\alpha)^p} {p^{p}} \sum_{n=1}^{\infty} |u_{n}|^{p} \ (n+1)^{\alpha-p} \geq \frac { 2^{\alpha-p} \, (p-1-\alpha)^p} {p^{p}} \sum_{n=1}^{\infty} |u_{n}|^{p} \ n^{\alpha-p} $$
This implies: $\mathcal{A}_{\alpha,p}^{\text{disc}} \geq 2^{\alpha-p} \mathcal{A}_{\alpha,p}^{\text{cont}}$. For the reverse inequality, we proceed as before by selecting $ \phi \in C_{0}^{\infty}(0,1)$ and defining $u_{n} = \phi(n/m)$ for some integer $m$, which leads to:
$$\mathcal{A}_{\alpha,p}^{\text{disc}} \leq  \inf_{\phi \in C^{\infty}_{0} (0,1)} \frac{\displaystyle \int_{0}^{1} |\phi'(x)|^p x^{\alpha} dx}{\displaystyle \int_{0}^{1} |\phi(x)|^p x^{\alpha-p} dx} = 
  \inf_{\phi \in C^{\infty}_{0} (0,\infty)} \frac{\displaystyle \int_{0}^{\infty} |\phi'(x)|^p x^{\alpha} dx}{\displaystyle \int_{0}^{\infty} |\phi(x)|^p  x^{\alpha-p} dx} = \mathcal{A}^{\text{cont}}_{\alpha,p} $$ 
Thus, the desired result is established.
\end{proof}

A natural question arises: do the sharp constants
$\mathcal{A}_{\alpha,p}^{\text{disc}}$ and
$\mathcal{A}_{\alpha,p}^{\text{cont}}$
coincide for $\alpha<0$ ?

\begin{remark}
For $p=2$ and $\alpha\in\mathbb{Z}^{-},$
the asymptotic behavior of the optimal Hardy weight suggests that the continuous sharp constant may fail to be admissible in the
discrete setting. Indeed, setting $\nu(n)=n^\alpha$, we have
\[
\mathcal{G}_{\nu}(n)
=
\sum_{k=1}^{n}k^{-\alpha}.
\]
Using Faulhaber's formula and the expression
of $w_{\nu}$ obtained in Section \ref{sec 3},
we obtain
\[
\frac{w_{\nu}(n)}{n^\alpha}
=
\frac{(\alpha-1)^2}{4n^2}
+
\frac{(\alpha-1)^2(\alpha-2)}{8n^3}
+
o\left(\frac{1}{n^3}\right).
\]
Since $\alpha<0$, the coefficient of
$n^{-3}$ is negative. Consequently,
for all sufficiently large $n$,
\[
w_{\nu}(n)
<
\mathcal{A}_{\alpha,2}^{\text{cont}}
n^{\alpha-2}.
\]
Although this asymptotic observation
does not by itself establish the failure
of the corresponding inequality, it
suggests a difference between the
discrete and continuous sharp constants.
The following proposition provides
a rigorous conclusion, valid for every
$p>1$.
\end{remark}

\begin{prop}\label{prop 4.11}
For every $p>1$ and $\alpha\leq-1$,
we have
\[
\mathcal{A}_{\alpha,p}^{\text{disc}}
<
\mathcal{A}_{\alpha,p}^{\text{cont}}.
\]
\end{prop}

\begin{proof}
For $N\geq1$, consider
$u_n=\mathbf{1}_{[1,N]}(n)$.
It follows that
$$
\mathcal{A}_{\alpha,p}^{\text{disc}} \left(\sum_{n=1}^{N}n^{\alpha-p}\right)
\leq
1+(N+1)^\alpha.
$$
Letting $N\to\infty$, we obtain
\[
\mathcal{A}_{\alpha,p}^{\text{disc}}
\leq
\frac{1}{\zeta(p-\alpha)}
<1
\leq
\left(\frac{p-1-\alpha}{p}\right)^p
=
\mathcal{A}_{\alpha,p}^{\text{cont}}.
\]
This completes the proof.
\end{proof}

 In what follows, we define $h_{\alpha}$ as the energy functional on $H_{0}^{1}(\mathbb{N})$, given by:  
$$ h_{\alpha}(u) = \sum_{n=1}^{\infty} |u_n - u_{n-1}|^p n^{\alpha} - \mathcal{A}_{\alpha,p}^{\text{disc}} \sum_{n=1}^{\infty} |u_{n}|^p n^{\alpha-p}.  
$$ 
Consider the function $\Psi : [0, \infty) \to [0, \infty)$, defined as $\Psi(t) = t^p$. We have the following result, ensuring the subcriticality of $h_{\alpha}$ as well as the stability of inequality \eqref{eq 1.3}.

\begin{thm} \label{thm 4.12} Let $\alpha \geq 0$ such that $\alpha \neq p-1.$ There exists a positive weight $\mathcal{R}_{\alpha}$ such that for all  $u \in H_{0}^{1}(\mathbb{N}):$ 
$$\sum_{n=1}^{\infty} |u_{n}-u_{n-1}|^{p} \ n^{\alpha}  -  \mathcal{A}^{\text{disc}}_{\alpha,p} \, \sum_{n=1}^{\infty} |u_{n}|^{p} \ n^{\alpha-p} \geq \sum_{n=1}^{\infty} |u_{n}|^{p} \mathcal{R}_{\alpha} (n)$$ 
In particular, the energy functional $h_{\alpha}$ is subcritical, and the inequality \eqref{eq 1.3} is $\Psi$-stable. \\
\end{thm}
\begin{proof} 
The case $\alpha = 0$ is a corollary of \cite{FKP23}[Theorem 1], where   
$$
\mathcal{R}_{\alpha}(n) = \left( 1 - \left( 1 - \frac{1}{n} \right)^{\frac{p-1}{p}} \right)^{p-1} - \left( \left( 1 + \frac{1}{n} \right)^{\frac{p-1}{p}} - 1 \right)^{p-1} - \left(\frac{p-1}{p}\right)^p \frac{1}{n^p},  
$$
which has been established to be positive. For $\alpha>0$ with $\alpha \neq p-1$, we proceed to demonstrate the existence of a positive weight $w_{\alpha}$ that satisfies:
\begin{equation}
\sum_{n=1}^{\infty} |u_{n}-u_{n-1}|^{p} \ n^{\alpha}  -  \mathcal{A}^{\text{disc}}_{\alpha,p} \, \sum_{n=1}^{\infty} |u_{n}|^{p} \ n^{\alpha-p} \geq \sum_{n=1}^{\infty} |u_{n}-u_{n-1}|^{p} w_{\alpha} (n), \tag{4.8} \label{*}
  \end{equation}
and the asymptotic behavior of $w_{\alpha}$ is given by:
$$
w_{\alpha}(n)
\underset{n\to\infty}{\asymp}
\begin{cases}
n^{\alpha-1}
& \text{if } \alpha \in(0,p-1),\\
n^{\alpha-2}
& \text{if } \alpha>p-1.
\end{cases}
$$
We start with the case $\alpha > p - 1$  Set for all $k \in \mathbb{N}:$
$$
c_{k}= \frac{(\alpha-p+1)^{2k}}{2^{2k}(p-1)^{2k} (2k+1)!}  
$$
and 
$$ h: n \to  
\sum_{k=0}^{\infty} \frac{c_k}{ n^{2k}}$$ 
Using Corollary \ref{cor 4.3}, we deduce that for every  $r\geq 2:$

$$ \left( \frac{1}{S_{r-1}} \right)^{\frac{1}{p-1}}- \left( \frac{1}{S_{r}} \right)^{\frac{1}{p-1}} \geq  \frac{(\alpha-p+1)^{\frac{p}{p-1}}}{(p-1) r^{\frac{\alpha}{p-1}}} \, h(r),$$
where $\displaystyle S_{r} = \sum_{k=1}^{r} k^{\alpha-p}.$ For $n \geq 1$, summing $r$ over the range $[n+1, \infty)$ yields:

$$ \sum_{r=n+1}^{\infty} \frac{h(r)}{r^{\frac{\alpha}{p-1}}} \leq  \frac{p-1}{(\alpha-p+1)^{\frac{p}{p-1}}} \left( \frac{1}{S_{n}} \right)^{\frac{1}{p-1}}
$$
 \\
For every $n\geq 1,$ set $f(n)= u_{n}-u_{n-1},$ $\mu(n)= n^{\alpha-p},$ and $\nu(n) = \displaystyle \frac{n^{\alpha}}{h^{p-1}(n)}.$  Then, $f \in F^{0}(\mathbb{N}^{*}).$ \\ \\
Consequently, we have $\min(B_{\mu,\nu}^{(1)},B_{\mu,\nu}^{(2)})  \leq \frac{(p-1)^{p-1}} {(\alpha-p+1)^p}$, and it follows that: \\

$$\sum_{n=1}^{\infty} |u_{n}-u_{n-1}|^{p} \frac{n^{\alpha}}{h^{p-1}(n)}  \geq \mathcal{A}^{\text{disc}}_{\alpha,p} 
 \sum_{n=1}^{\infty} |u_{n}|^p n^{\alpha-p}$$
Define $$w_{\alpha}: n \to n^{\alpha} \left( 1 - \frac{1}{h^{p-1}(n)} \right) > 0$$ 

As a result, inequality \eqref{*} holds,
and $w_{\alpha}(n) \asymp n^{\alpha-2}$
as $n\to\infty$. \\ \\
If $0 < \alpha < p-1,$ define: $\nu(n) = \big(n- \frac{1}{2}\big)^{\alpha}, \, \, \, n\geq 1$ and we have from \cite{Liu12}:
$$\sum_{n=1}^{\infty} |u_{n}-u_{n-1}|^{p} \nu(n) \geq \mathcal{A}^{\text{disc}}_{\alpha,p} \, \sum_{n=1}^{\infty} |u_{n}|^{p} \ n^{\alpha-p} $$  
then for $w_{\alpha}: n \to n^{\alpha}- \nu(n) > 0 $  the inequality \eqref{*} holds, and
\[
w_{\alpha}(n) \asymp n^{\alpha-1}
\qquad \text{as } \,n\to\infty.
\]
In both cases, through Corollary \ref{cor 2.4}, it follows that there exists a positive weight $\mathcal{R}_{\alpha}$ such that: 
$$\sum_{n=1}^{\infty} |u_{n} - u_{n-1}|^{p} w_{\alpha}(n) \geq \sum_{n=1}^{\infty} |u_n|^p \mathcal{R}_{\alpha}(n).
$$
Thus, the energy functional $h_{\alpha}$ is subcritical.
We equip $H_0^{1}(\mathbb{N})$ with the weighted norm
\[
\|u\|_{\mathcal{R}_{\alpha}}
=
\left(
\sum_{n=1}^{\infty}
|u_n|^p\mathcal{R}_{\alpha}(n)
\right)^{1/p},
\]
and consider the associated metric
$d(u,v)=\|u-v\|_{\mathcal{R}_{\alpha}}.$
Since $\mathcal{R}_{\alpha}$ is positive, the set of optimizers reduces to
the zero sequence. Consequently,
\[
h_{\alpha}(u)
\geq
\|u\|_{\mathcal{R}_{\alpha}}^p
=
\Psi(d(u,0)),
\]
where $\Psi(t)=t^p$. Therefore, inequality \eqref{eq 1.3}
is $(d,\Psi)$-stable.
This completes the proof.
\end{proof}
We end this section with the proof of Lemma \ref{lemma 4.2}.
\begin{proof} Consider the set: $\mathcal{U} = \left\{ \frac{1}{n} \mid n \in \mathbb{N}^{*} \right\}.$ The proof relies on showing that: for all
$(x,\gamma)\in \mathcal{U} \times [2,\infty[:$  $ \frac{\partial F(x,\gamma)}{\partial \gamma} \geq 0,$
which implies $F(x,\gamma) \geq  F(x,2).$ Therefore, it remains to prove that $F(x,2) \geq 0$ for all $x \in \mathcal{U}.$ 
However, the derivatives on $[0,1]$ are given by:
$$F(x,2)= (1+x)(1-\exp(-2x))- \exp\left(\frac{2x}{1+x}\right) +1 $$

$$\frac{\partial F(x,2)}{\partial x} = \exp(-2x)(2x+1)+1 - \frac{2}{(1+x)^2} \exp\left(\frac{2x}{1+x}\right)$$

$$\frac{\partial^2 F(x,2)}{\partial x^2} = \frac{4 x \exp(-2x) \left[ \exp\left(\frac{2x(x+2)}{1+x}\right)- (1+x)^4 \right]} {(1+x)^4}$$ 
In order to prove that $\forall x \in [0,1]:$ $(1+x)^4 \leq \exp\left(\frac{2x(x+2)}{1+x}\right)$ $\Leftrightarrow $ 
$2(1+x) \log(1+x) \leq x(x+2),$ one can consider the function $g: x \to x(x+2)- 2(1+x)\log(1+x),$ where, for all $x \in [0,1]:$ \\
$g'(x)=2(x-\log(1+x)) \geq 0,$ therefore $g(x) \geq g(0)=0.$ Consequently, we have $\frac{\partial^2 F(x,2)}{\partial x^2} \geq 0,$ which leads to $\frac{\partial F(x,2)}{\partial x} \geq \frac{\partial F}{\partial x} (0,2) = 0,$ this implies that $F(x,2) \geq F(0,2) = 0.$ \\ \\
We now return to the assertion that $\frac{\partial F(x,\gamma)}{\partial \gamma} \geq 0.$ To proceed, recall the weighted power main inequality. Let $a,b$ be positive reals and $\alpha \in [0,1].$ For any real $r$, we define:
$$
\mathcal{M}_{r}^{\alpha}(a,b) =
\begin{cases}
\left( \alpha a^r + (1-\alpha) b^r \right)^{1/r}, & \text{if } r \neq 0, \\
a^{\alpha} b^{1-\alpha}, & \text{if } r = 0.
\end{cases}
$$
If $r > s$, then $\mathcal{M}_{r}^{\alpha}(a,b) \geq \mathcal{M}_{s}^{\alpha}(a,b)$. Equality occurs if and only if $a=b$ or $\alpha \in \left\{0,1\right\}$. \\ \\
First, we have for every $(x,\gamma) \in \mathcal{U} \times [2,\infty[:$
$$\frac{\partial F(x,\gamma)}{\partial \gamma} =
\frac{x}{1+x}  \exp \left( \frac{\gamma x}{1+x} \right) \left[ (1+x)^{\gamma} \exp \left( \frac{-\gamma x}{1+x}\right) \left( \frac{\log(1+x)}{x} + \frac{x- \log(1+x)}{x} \exp(-\gamma x) \right) - 1 \right] $$
Set $ a = (1 + x) \exp\left(\frac{-x}{1 + x}\right) $, $ b = (1 + x) \exp\left(\frac{-x}{1 + x} - x\right) $, and $ \alpha = \frac{\log(1 + x)}{x}.$ It follows that:

$$\frac{\partial F(x,\gamma)}{\partial \gamma} =
\frac{x}{1+x}  \exp \left( \frac{\gamma x}{1+x} \right) \left[ \left( \mathcal{M}_{\gamma}^{\alpha} (a,b)\right)^{\gamma} -1 \right]$$
As a direct result of the weighted power mean inequality, we obtain:
$$\frac{\partial F(x,\gamma)}{\partial \gamma} \geq 
\frac{x}{1+x}  \exp \left( \frac{\gamma x}{1+x} \right) \left[ \left( \mathcal{M}_{2}^{\alpha} (a,b)\right)^{\gamma} -1 \right]
$$
Thus, demonstrating that $\mathcal{M}_{2}^{\alpha} (a,b) \geq 1$ will lead to  $\frac{\partial F(x,\gamma)}{\partial \gamma} \geq 0.$ Also, $\mathcal{M}_{2}^{\alpha} (a,b)$ is given by:
$$ \mathcal{M}_{2}^{\alpha} (a,b)
= (1+x) \exp \left( \frac{-x}{1+x} \right) \left( \frac{\log(1+x)}{x} + \frac{x-\log(1+x)}{x} \exp(-2x) \right)^{\frac{1}{2}}
$$
We are thus reduced to an inequality involving one variable. Nevertheless, one has: 
$$ \mathcal{M}_{2}^{\alpha} (a,b) \geq 1  \Leftrightarrow  \log(1+x)
-\frac{x}{1+x} + \frac{1}{2} \log \left(  \frac{\log(1+x)}{x} (1-\exp(-2x)) +\exp(-2x) \right) \geq 0
$$
Using the concavity of the logarithm, it suffices to prove that for all $x \in \mathcal{U}:$
$$\log(1+x)
-\frac{x}{1+x} +  \frac{1-\exp(-2x)}{2} \log \left(  \frac{\log(1+x)}{x} \right) \geq 0 $$
One can verify that the inequality holds for $x = \frac{1}{n},$ where $ n \in \{1, 2\}.$ However, on the interval $(0, \frac{1}{3}],$ we start by noting that:
$$ \log \left(  \frac{\log(1+x)}{x} \right)  =  \log \left(  \sum_{n=0}^{\infty} \frac{(-1)^{n}\, x^{n}}{n+1} \right) \geq \log \left(  \sum_{n=0}^{5} \frac{(-1)^{n}\, x^{n}}{n+1} \right)$$
and
$$h_{1}: x \to \log \left(  \sum_{n=0}^{5} \frac{(-1)^{n}\, x^{n}}{n+1} \right) - \left( \frac{-x}{2} + \frac{5x^2}{24} -\frac{x^3}{8} \right)$$
we have:
$$h_{1}'(x) =\frac{x^3 \left(90 x^4 - 208 x^3 + 375 x^2 + 726 x - 502\right)}{24 \left(10 x^5 - 12 x^4 + 15 x^3 - 20 x^2 + 30 x - 60\right)}
$$
with the change of variables: $u = \frac{1}{3x} - 1 \Leftrightarrow x = \frac{1}{3(u + 1)}.$ Hence, for $x\in (0,\frac{1}{3}]:$ $ u \geq 0 $ and:

$$h_{1}'(x)=\frac{13\,554 u^4 + 47\,682 u^3 + 60\,597 u^2 + 32\,572 u + 6\,073}{72 (u + 1)^2 \left(14\,580 u^5 + 70\,470 u^4 + 136\,620 u^3 + 132\,705 u^2 + 64\,566 u + 12\,581\right)} > 0
$$
therefore $h_{1}$ is non decreasing and for all $x \in (0,\frac{1}{3}]$ :
$h_{1}(x) \geq h_{1}(0) = 0,$ then on $(0,\frac{1}{3}]$:

$$\log \left(  \frac{\log(1+x)}{x} \right) \geq \frac{-x}{2} + \frac{5x^2}{24} -\frac{x^3}{8}
$$
Note that:

$$\frac{-x}{2} + \frac{5x^2}{24} - \frac{x^3}{8} = \frac{-x}{24} \left( 3\left(x - \frac{5}{6}\right)^2 + \frac{119}{12} \right) \leq 0$$
Using the fact that, for all $ u \in \mathbb{R} $:
$$
\exp(u) \geq 1 + u + \frac{u^2}{2} + \frac{u^3}{6}.
$$
Accordingly, setting
$$ h_{2}: x \to  \log(1+x)
-\frac{x}{1+x} +  \left( x-x^2+ \frac{2x^3}{3} \right)  \left( \frac{-x}{2} + \frac{5x^2}{24} -\frac{x^3}{8} \right)$$ 
\\
It suffices to prove that, for all $x \in (0,\frac{1}{3}],$ $h_{2}(x) \geq 0,$ However the derivative of $h$ gives us: \\
$$h_{2}'(x) = \frac{x^2 \left( -36 x^5 + 23 x^4 - 38 x^3 - 136 x^2 + 42 x + 9 \right)}{72 (1 + x)^2}
$$
As before, we apply the change of variables: $u = \frac{1}{3x} - 1 \Leftrightarrow x = \frac{1}{3(u + 1)}.$ Thus $u \geq 0,$ and:
$$-36 x^5 + 23 x^4 - 38 x^3 - 136 x^2 + 42 x + 9 = \frac{729 u^5 + 4779 u^4 + 10602 u^3 + 10308 u^2 + 4304 u + 536}{81 (u + 1)^5} > 0.$$
Therefore, $h$ is non decreasing on $(0,\frac{1}{3}]$ which implies, $h_{2}(x) \geq h_{2}(0)=0.$ Thus, $\mathcal{M}_{2}^{\alpha} (a,b) \geq 1,$ and this concludes the proof of  Lemma \ref{lemma 4.2}.
\end{proof}
\section*{Acknowledgements} 
\phantom{.}\\
The author thanks Cyril Roberto and Ari Laptev for their insightful discussions and valuable suggestions, which have contributed to this research.

\bibliographystyle{plain}
\bibliography{mainfile.bib}

@article{KPP18,
  author    = {M. Keller and Y. Pinchover and F. Pogorzelski},
  title     = {Optimal Hardy inequalities for Schrödinger operators on graphs},
  journal   = {Communications in Mathematical Physics},
  volume    = {358},
  year      = {2018},
  pages     = {767--790},
  doi       = {10.1007/s00220-018-3110-6}
}

@article{DP16,
  author    = {B. Devyver and Y. Pinchover},
  title     = {Optimal $L^p$ Hardy-type inequalities},
  journal   = {Annales de l'Institut Henri Poincaré. Analyse Non Linéaire},
  volume    = {33},
  number    = {1},
  year      = {2016},
  pages     = {93--118},
  doi       = {10.1016/j.anihpc.2015.01.005}
}

@article{Liu12,
title = {Some generalizations and improvements of discrete Hardy’s inequality},
journal = {Computers \& Mathematics with Applications},
volume = {63},
number = {3},
pages = {601-607},
year = {2012},
issn = {0898-1221},
doi = {https://doi.org/10.1016/j.camwa.2011.09.019},
url = {https://www.sciencedirect.com/science/article/pii/S0898122111007875},
author = {Jianzhong Liu and Xuande Zhang and Bo Jiang},}

@misc{Fis22,
  title     = {Quasi-Linear Criticality Theory and Green's Functions on Graphs},
  author    = {F. Fischer},
  year      = {2022},
  eprint    = {2207.05445},
  archivePrefix = {arXiv},
  primaryClass = {math-ph},
  url       = {https://arxiv.org/abs/2207.05445}
}

@article{FKP23,
  author    = {F. Fischer and M. Keller and F. Pogorzelski},
  title     = {An improved discrete $p$-Hardy inequality},
  journal   = {Integral Equations and Operator Theory},
  volume    = {95},
  number    = {4},
  year      = {2023},
  pages     = {Paper No. 24, 17 pages},
  doi       = {10.1007/s00020-023-02646-4}
}

@article{Fis24,
  author    = {F. Fischer},
  title     = {On the optimality and decay of $p$-Hardy weights on graphs},
  journal   = {Calculus of Variations and Partial Differential Equations},
  volume    = {63},
  pages     = {162},
  year      = {2024},
  doi       = {10.1007/s00526-024-02754-0}
}

@article{Car17,
  author    = {E.A. Carlen},
  title     = {Duality and Stability for Functional Inequalities},
  journal   = {Annales de la Faculté des sciences de Toulouse : Mathématiques, Série 6},
  volume    = {26},
  number    = {2},
  pages     = {319--350},
  year      = {2017},
  doi       = {10.5802/afst.1535},
  url       = {https://afst.centre-mersenne.org/articles/10.5802/afst.1535/}
}

@article{KPP20,
  author    = {M. Keller and Y. Pinchover and F. Pogorzelski},
  title     = {Criticality theory for Schrödinger operators on graphs},
  journal   = {Journal of Spectral Theory},
  volume    = {10},
  number    = {1},
  pages     = {73--114},
  year      = {2020},
  doi       = {10.4171/JST/286}
}

@article{Muck72,
  author    = {B. Muckenhoupt},
  title     = {Hardy's inequality with weights},
  journal   = {Studia Mathematica},
  volume    = {44},
  number    = {1},
  pages     = {31--38},
  year      = {1972},
  url       = {http://eudml.org/doc/217718},
  language  = {eng}
}

@article{Mic99,
  author    = {L. Miclo},
  title     = {An example of application of discrete Hardy's inequalities},
  journal   = {Markov Processes and Related Fields},
  volume    = {5},
  number    = {3},
  pages     = {319--330},
  year      = {1999}
}

@article{Mic98,
  author    = {L. Miclo},
  title     = {Relations entre isopérimétrie et trou spectral pour les chaînes de Markov finies},
  journal   = {Probability Theory and Related Fields},
  volume    = {114},
  pages     = {431--485},
  year      = {1999},
  doi       = {10.1007/s004400050231},
  url       = {https://doi.org/10.1007/s004400050231}
}

@article{Fig12,
  author    = {A. Figalli},
  title     = {Stability in geometric and functional inequalities},
  journal   = {Congress of Mathematics},
  year      = {2013},
  note      = {CVGMT preprint},
  url       = {http://cvgmt.sns.it/paper/1872/}
}

@article{Evans95,
  author    = {W. D. Evans and D. J. Harris and L. Pick},
  title     = {Weighted Hardy and Poincaré inequalities on trees},
  journal   = {Journal of the London Mathematical Society},
  volume    = {52},
  number    = {2},
  pages     = {121--136},
  year      = {1995}
}

@article{Gup22,
  author    = {S. Gupta},
  title     = {Discrete weighted Hardy inequality in 1-D},
  journal   = {Journal of Mathematical Analysis and Applications},
  volume    = {514},
  number    = {2},
  pages     = {126345},
  year      = {2022},
  month     = {October},
  doi       = {10.1016/j.jmaa.2022.126345},
  url       = {http://dx.doi.org/10.1016/j.jmaa.2022.126345}
}

@article{Ye23,
  author    = {X. Huang and D. Ye},
  title     = {One-dimensional sharp discrete Hardy–Rellich inequalities},
  journal   = {Journal of the London Mathematical Society},
  volume    = {109},
  number    = {1},
  year      = {2023},
  month     = {November},
  doi       = {10.1112/jlms.12838},
  url       = {http://dx.doi.org/10.1112/jlms.12838},
  issn      = {1469-7750},
  publisher = {Wiley}
}

@misc{SB18,
  title     = {The Hermite-Hadamard inequality in higher dimensions},
  author    = {S. Steinerberger},
  year      = {2018},
  eprint    = {1808.07794},
  archivePrefix = {arXiv},
  primaryClass = {math.CA},
  url       = {https://arxiv.org/abs/1808.07794}
}

@misc{Har20,
  author    = {G. H. Hardy},
  title     = {Note on a theorem of Hilbert},
  journal   = {Mathematische Zeitschrift},
  volume    = {6},
  pages     = {314--317},
  year      = {1920},
  doi       = {10.1007/BF01199965},
  url       = {https://doi.org/10.1007/BF01199965}
}

@misc{Lan21,
  author    = {E. Landau},
  title     = {A letter to G. H. Hardy, June 21, 1921},
  year      = {1921}
}

@article{FigInd13,
  author    = {A. Figalli and E. Indrei},
  title     = {A Sharp Stability Result for the Relative Isoperimetric Inequality Inside Convex Cones},
  journal   = {Journal of Geometric Analysis},
  volume    = {23},
  pages     = {938--969},
  year      = {2013},
  doi       = {10.1007/s12220-011-9270-4},
  url       = {https://doi.org/10.1007/s12220-011-9270-4}
}

@article{Fig10,
  author    = {A. Figalli and F. Maggi and A. Pratelli},
  title     = {A mass transportation approach to quantitative isoperimetric inequalities},
  journal   = {Inventiones Mathematicae},
  volume    = {182},
  number    = {1},
  pages     = {167--211},
  year      = {2010}
}

@article{Fig09,
  author    = {A. Figalli and F. Maggi and A. Pratelli},
  title     = {A refined Brunn-Minkowski inequality for convex sets},
  journal   = {Annales de l'Institut Henri Poincaré. Analyse Non Linéaire},
  volume    = {26},
  number    = {6},
  pages     = {2511--2519},
  year      = {2009}
}

@article{FigMagPra13,
  author    = {A. Figalli and F. Maggi and A. Pratelli},
  title     = {Sharp stability theorems for the anisotropic Sobolev and log-Sobolev inequalities on functions of bounded variation},
  journal   = {Advances in Mathematics},
  volume    = {242},
  pages     = {80--101},
  year      = {2013},
  issn      = {0001-8708},
  doi       = {10.1016/j.aim.2013.04.007},
  url       = {https://doi.org/10.1016/j.aim.2013.04.007}
}

@misc{Caz20,
  title     = {The method of super-solutions in Hardy and Rellich type inequalities in the $L^2$ setting: an overview of well-known results and short proofs},
  author    = {C. Cazacu},
  year      = {2020},
  eprint    = {2003.11798},
  archivePrefix = {arXiv},
  primaryClass = {math.AP},
  url       = {https://arxiv.org/abs/2003.11798}
}

@article{Bian18,
  title     = {A note on the Sobolev inequality},
  journal   = {Journal of Functional Analysis},
  volume    = {100},
  number    = {1},
  pages     = {18--24},
  year      = {1991},
  issn      = {0022-1236},
  doi       = {https://doi.org/10.1016/0022-1236(91)90099-Q},
  url       = {https://www.sciencedirect.com/science/article/pii/002212369190099Q},
  author    = {G. Bianchi and H. Egnell}
}

@misc{Ver21,
  title     = {Optimal Hardy-weights for the $(p,A)$-Laplacian with a potential term}, 
  author    = {I. Versano},
  year      = {2021},
  eprint    = {2112.04449},
  archivePrefix = {arXiv},
  primaryClass = {math.AP},
  url       = {https://arxiv.org/abs/2112.04449}
}
\end{document}